\documentclass[a4paper,11pt]{amsart}
\usepackage[utf8]{inputenc}
\usepackage{cite}
\usepackage{url}
\usepackage{xcolor}
\usepackage{bbm}
\usepackage[foot]{amsaddr}
\newtheorem{theorem}{Theorem}[section]
\newtheorem{lemma}[theorem]{Lemma}
\newtheorem{definition}[theorem]{Definition}
\newtheorem{assumption}{Assumption}
\newtheorem{proposition}[theorem]{Proposition}
\newtheorem{remark}[theorem]{Remark}
\newtheorem{corollary}[theorem]{Corollary}
\newtheorem{example}[theorem]{Example}

\allowdisplaybreaks[4]
\usepackage{enumerate}

\numberwithin{equation}{section}
\usepackage{hyperref}

\newcommand{\Bcal} {{\mathcal B}}

\newcommand{\Fcal} {{\mathcal F}}

\newcommand{\Mcal} {{\mathcal M}}

\newcommand{\Mfrak} {{\mathfrak M}}

\newcommand{\R}{\mathbb{R}}
\newcommand{\N}{\mathbb{N}}

\renewcommand{\P}{\mathbb{P}}
\newcommand{\E}{\mathbb{E}}

\renewcommand{\epsilon}{\varepsilon}
\newcommand{\eps}{\varepsilon}

\newcommand{\Lip}{\operatorname{Lip}}
\newcommand{\loc}{\operatorname{loc}}

\newcommand{\ud}{\mathrm{d}}

\usepackage{geometry}
\usepackage{enumerate}
\usepackage{graphicx}
\usepackage{tikz}
\usepackage{atbegshi}

\AtBeginShipoutFirst{\begin{tikzpicture}[remember picture, overlay]
        \node[anchor=north west, xshift=1in, yshift=-1.5cm] at (current page.north west) {\includegraphics[width=2cm]{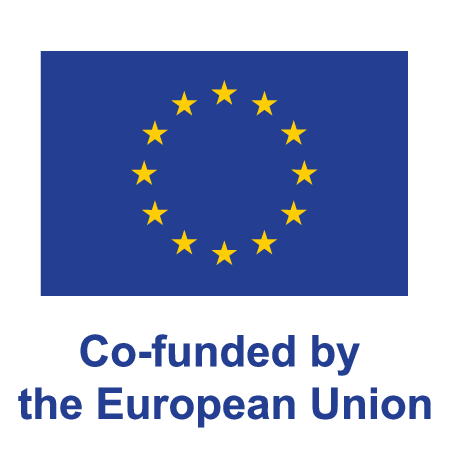}};
    \end{tikzpicture}}

\title[Ergodicity for stochastic neural field equations]{Ergodicity for stochastic neural field equations}
\author[A.-M. Otsetova]{Anna-Mariya Otsetova}
\email{anna-mariya.otsetova@aalto.fi}
\author[J. M. T\"olle]{Jonas M. T\"olle}
\email{jonas.tolle@aalto.fi}
\address[AMO, JMT]{Aalto University, Department of Mathematics and Systems Analysis, PO Box 11100 (Otakaari 1, Espoo), 00076 Aalto, Finland. ORCID AMO: 0009-0006-0420-2728, ORCID JMT: 0000-0002-6020-4131}

\date{\today}


\begin{document}

\begin{abstract}
We investigate the well-posedness and long-time behavior of a general continuum neural field model with Gaussian noise on possibly unbounded domains. In particular, we give conditions for the existence of invariant probability measures by restricting the solution flow to an invariant subspace with a nonlocal metric. Under the assumption of a sufficiently large decay parameter relative to the noise intensity, the norm of the connectivity kernel operator, and the Lipschitz regularity of the activation function, we establish exponential ergodicity and exponential mixing of the associated Markovian Feller semigroup and the uniqueness of the invariant measure.
\end{abstract}

\thanks{This work is licensed under the Creative Commons Attribution 4.0 International License. To view a copy of this license, visit \url{http://creativecommons.org/licenses/by/4.0/} or send a letter to Creative Commons, PO Box 1866, Mountain View, CA 94042, USA. \includegraphics[height=1em]{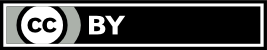}
Original work to appear in Applied Mathematics \& Optimization (2026), 38 pp., \url{https://doi.org/10.1007/s00245-026-10537-6}.}

\maketitle

{\footnotesize
\tableofcontents
}

\section{Introduction}

Mathematical neuroscience, that is, mathematical modeling of neurobiological processes on different scales, has gained a lot of attention in recent decades, notably because of its striking analogy to artificial neural networks, prominent in machine learning.
In 1977, Amari \cite{Amari} proposed a general continuum neural field model which already features many of the observed dynamic pattern-formation effects\footnote{Pattern formation is classically linked to reaction-diffusion systems, first studied by Turing in \cite{T:52}.} in neuroscience, see \cite{neuralfields2014ch9} for an overview of neural field models. The nonlinear \emph{Amari neural field equation} models the voltage $v$ in a continuum mean-field model of the neural cortex $U\subset\mathbb{R}^d$,
\[\partial_t v=-\alpha v+K f(v),\]
where $\alpha>0$ is a decay parameter, $w:U\times U\to\R$ models the connectivity of neurons via the kernel operator,
\[(K u)(x):=\int_U w(x,y)u(y)\,dy,\quad \text{for a.e.}\; x\in U,\]
and $f:\R\to\R$ is a neural activation function, typically as follows.
\begin{example}\label{ex:activation}
Commonly used activation functions include:
\begin{description} 
\item[ReLU] $f(x)=x\vee 0$;
\item[Heaviside] $f(x)=\mathbbm{1}_{[0,\infty)}(x)$;
\item[Logistic] $f(x)=(1+e^{-x})^{-1}$;
\item[Hyperbolic tangent] $f(x)=\tanh(x)$.
\end{description}
\end{example}
Existence and uniqueness of solutions for this model have been studied for instance in \cite{PG:10}.

Several stochastic generalizations have emerged since the seminal contributions of Amari \cite{Amari} and Wilson and Cowan \cite{WilsonCowan}.
The \emph{stochastic neural field model} can be derived from scaling arguments in different ways.
On the one hand, it can be obtained heuristically by interpreting the (small) Gaussian noise as quantum or thermal fluctuations, or measurement errors. For an overview and heuristic derivation of the Amari model from microscopic neurobiological models, we refer to the comprehensive articles \cite{Bressloff} and \cite{CR:25}, and to the book \cite{neuralfields2014ch9}. On the other hand, stochastic continuum models emerge naturally as scaling limits of finite-dimensional models for interacting neurons subject to spiking, which is modeled as random pure jump noise.

As a motivational example, we consider the model from \cite{ELL:21,ELL:23},
in which $N$ interacting neurons are represented by their membrane potentials and structured into $P$ populations. For any $ 1 \le k \le P$, let $ X_t^{N, k, i }, 1 \le i \le N_k $, be the $N_k$ membrane potential processes representing the neurons within the $k$th population, where $\sum_{k=1}^P N_k=N$. The system evolves according to the following pure jump stochastic dynamics for a finite time horizon $T>0$.
\begin{equation}\label{eq:fini}
d X_t^{N, k, i } = - \alpha X_t^{N, k, i } \,dt + \sum_{l=1}^P \tilde{w}(k, l ) \frac{1}{N_l} \sum_{j=1}^{N_l} \int_{0}^\infty \mathbbm{1}_{\{ z \le f ( X_{t-}^{N, l , j } ) \}} \,\pi^{l, j } ( dt, dz ) ,\quad t\in (0,T],
\end{equation} 
where, for each $k$, the $ X^{N, k, i}_0, 1 \le i \le N_k $, are independent and identically distributed (i.i.d.) according to some initial probability measure $X_0^{N,k,i}\sim \nu_0^k $, and where $ \pi^{l, j} , 1  \le l\le P, 1 \le j \le N_l $, are i.i.d. Poisson random measures on $ [0,\infty)\times [0,\infty) $ having Lebesgue intensity. 

The authors consider the following, more specific continuum model on $U=[0,1]$,
\begin{equation}\label{eq:specialmodel}d u(x,t) = - \alpha u(x,t) \,dt + \int_U w( x, y) f( u(y,t) ) \,dy \,dt + \frac{1}{\sqrt{N}} \int_U w(x,y) \sqrt{ f ( u(y,t) ) }\, W (dt, dy) ,\end{equation}
where $W$ is space-time white noise on $[0,T]\times [0,1]$, mollified in the space component by the interaction kernel.
The convergence to this model is obtained via Koml\'os--Major--Tusn\'ady (KMT) coupling \cite{KMT} by a space-time rescaling of the finite-dimensional pure jump stochastic differential equations \eqref{eq:fini},
where each population corresponds to the box $U=[0,1]$ such that $ \tilde{w} (k, l ) = w( \frac{k}{P}, \frac{l}{P} )$, for any $ 1 \le k , l \le P $.

For specific $f$, e.g., if $f$ is the logistic function, we obtain exponential ergodicity for \eqref{eq:specialmodel} under certain assumptions on the kernel\footnote{This requires, e.g., $\sup_{y\in U}\int_U |w(x,y)|^2\,dx<\infty$.} $w$, and on $\alpha$.
See \cite{CMT:21,CCDRB:15,CDLO:19,CO:20} for related approaches via Hawkes processes. Ergodicity for the finite-dimensional system has been obtained in \cite{DLL:20}.

We choose to work with the more general model on $U\subset\R^d$, where $U$ is the closure of a non-empty, possibly unbounded open domain in $\R^d$. We fix $T>0$ to be a finite time horizon. For a decay parameter $\alpha>0$, a connectivity kernel $w:U\times U\to\R$, and a nonlinear activation function $f:\R\to\R$, we study the following general stochastic integral equation with Gaussian Wiener noise
\begin{align}\label{eq:gen_model}
\begin{split}
du(\cdot,t) = - \alpha u(\cdot,t) \,dt +& \int_U w( \cdot, y) f( u(y,t) )\, dy\, dt 
+ B (u(\cdot,t))\, dW(t) ,\quad t\in (0,T],\\
&u(\cdot,0)=u_0(\cdot)\in L^2(U,\rho).
\end{split}
\end{align}
Here, $\rho> 0$ $dx$-a.e., $\rho\in L^1_{\loc}(U)$ is a weight function and we set $\nu(dx):=\rho(x)\, dx$. Furthermore, $\{W(t)\}_{t\ge 0}$ denotes a cylindrical Wiener process\footnote{See \cite{DaPrZa:2nd,LR:15} for the notion of a cylindrical Wiener process. A very brief explanation of this notion in $L^2$-spaces is provided in \cite[Appendix A]{KuehnToelle19}.} in a separable Hilbert space $V$, modeled on a filtered probability space $(\Omega,\Fcal,\{\Fcal_t\}_{t\ge 0},\P)$ satisfying the usual conditions. The noise coefficient $B:L^2(U,\rho)\to L_2(V,L^2(U,\rho))$ takes values in a space of Hilbert--Schmidt operators.

For simplicity, we denote $u(t):=u(\cdot,t)\in L^2(U,\rho)$. We may also denote $u(t;u_0):=u(\cdot,t;u_0):=u(t)$, whenever $u(0)=u_0\in L^2(U,\rho)$.

Stochastic neural field models feature many interesting phenomena, such as
traveling wave solutions \cite{IM:16,K:20},
traveling bumps and wandering patterns \cite{PollKilpatrick,KilpatrickErmentrout,BK:15}, and front propagation \cite{KruegerStannat}.
Stability of neural field equations has been studied in \cite{VeltzFaugeras2010,CRS:25}. The singular case of a Heaviside activation function has been studied in \cite{Kruger:2017eb}.

The non-negative semidefiniteness, or non-positive semidefiniteness, respectively, of the kernel $w$ (see Assumption \ref{assu:hat}) can be interpreted as domination of excitation, or inhibition effects in the neural field, respectively, as observed in perceptual bistability within binocular rivalry, see \cite{vanEe,WebberBressloff,Moreno-BoteRinzelRubin}. Our main ergodicity result is obtained in either of these cases, however, the mixed case remains an open problem.
In the case of inhibition and a monotone activation function, we obtain the improved ergodicity result Theorem \ref{thm:monotone} in Subsection \ref{subsec:4.4}.

We obtain the following main result. The assumptions are discussed in Section \ref{sec:2}.
\begin{theorem}\label{thm:main}\phantom{ }
\begin{enumerate}
\item
Under Assumptions \ref{assu:space} and \ref{assu:noise}, for any finite time horizon $T>0$ and initial data in $L^2(U,\rho)$, there exists a unique $\P$-a.s. continuous and $\{\mathcal{F}_t\}_{t\in [0,T]}$-adapted strong solution to \eqref{eq:gen_model} in $L^2(U,\rho)$, see Definition \ref{def:sol} for the terminology.
\item
Under Assumptions \ref{assu:space}, \ref{assu:noise}, \ref{assu:hat}, and \ref{assu:inv}, there exists an invariant probability measure for the Markovian Feller semigroup associated to \eqref{eq:gen_model}, see Definition \ref{def:inv} for the terminology.
\item
Under Assumptions \ref{assu:space}, \ref{assu:noise}, \ref{assu:hat}, \ref{assu:inv}, and \ref{assu:ergo}, there exists a unique invariant probability measure for the exponentially ergodic and exponentially mixing Markovian Feller semigroup associated to \eqref{eq:gen_model}, see Definitions \ref{def:inv} and \ref{def:ergo} for the terminology.
\end{enumerate}
\end{theorem}
The items of Theorem \ref{thm:main} follow from Theorem \ref{thm:existence_all_cases}, Theorem \ref{thm:existence}, and Theorem \ref{thm:uniqueness} below, respectively, which also contain all technical details.
Note that part (i) of Theorem \ref{thm:main} has already, up to minor improvements, essentially been proved in \cite{KuehnRiedler,FaugerasInglis}. We are including it here to make our presentation more concise. We would like to point out that our situation is slightly more general, as it includes both the additive and the multiplicative noise cases, on a bounded domain as well as on a possibly unbounded domain within a weighted space. More importantly, we do not assume that the connectivity kernel $w$ is symmetric. The main contribution of this work is the existence of a unique invariant measure and the exponential ergodicity under the conditions discussed below. Our main method relies on the invariance of a compactly embedded nonlocal Hilbert subspace, based on the nonlocal transformation of the ambient space introduced by Kuehn and the second author in \cite{KuehnToelle19}, and the Krylov--Bogoliubov method \cite{DPZ:96}. Uniqueness and exponential ergodicity are then obtained by similar ideas as in
\cite{ESSTvG,BDP}. The main Assumption \ref{assu:ergo} (see Subsection \ref{subsec:ergoassu}) for the uniqueness of invariant measures is the following quantitative relation between the kernel $w$, the kernel operator $K$, the activation function $f$, the noise coefficient $B$, and the decay rate $\alpha>0$.
\[\|K\|_{L(L^2(U,\rho),L^2(U,\rho))}\Lip(f)+\frac{1}{2}C_B<\alpha,\]
where $C_B\ge 0$ is as in Assumption \ref{assu:noise} below. The condition resembles the assumptions for the exponential frame in \cite{DLL:20,ESSTvG,BDP} and displays the natural hierarchy of parameters of the model, see also \cite{kulik2017ergodic}.

It would be interesting to obtain similar results, when the kernel is neither non-negative nor non-positive semidefinite, relating to mixed inhibition and excitation. Another interesting open problem is the analysis of metastability and Kramers' law \cite{BerglundGentz13} which could be based on the gradient structure introduced in \cite{KuehnToelle19}.

\subsection*{Notation}
For a \emph{Polish space} $X$, i.e., a separable and completely metrizable topological space, denote by $\mathcal{B}(X)$ the Borel $\sigma$-algebra of $X$, that is, the smallest $\sigma$-algebra containing all open subsets of $X$. Let $\mathcal{M}_1(X)$ denote the set of all probability measures on $\mathcal{B}(X)$ (called \emph{Borel probability measures}). If $X$ is a normed space with norm $\|\cdot\|_X$, for $p\ge 1$, let $\mathcal{M}_1^p(X)$ denote the set of all Borel probability measures with \emph{finite $p$-moment}, that is, all $\mu\in\mathcal{M}_1(X)$ such that $\int_X \|z\|^p_X\,\mu(dz)<\infty$. Denote by $\mathcal{B}_b(X)$ the space of bounded Borel measurable functions from $X$ to $\R$. Denote by $C_b(X)$ the space of all continuous and bounded functions from $X$ to $\R$, equipped with the supremum norm $\|\varphi\|_\infty:=\sup_{x\in X}|\varphi(x)|$, and denote by $\Lip_b(X)$ the space of bounded and Lipschitz continuous functions from $X$ to $\R$ with norm
$\|\varphi\|_{\Lip_b} :=  \Lip(\varphi) + \|\varphi\|_{\infty}$, where $\Lip(\varphi)$ denotes the smallest constant $L\ge 0$ such that
\[|\varphi(x)-\varphi(y)|\le L\|x-y\|_X\quad\text{for every}\;\;x,y\in X.\]
For real numbers $a,b\in\R$, denote $a\vee b:=\max(a,b)$ and $a\wedge b:=\min(a,b)$. For a measurable subset $U\subset\R^d$ with finite Lebesgue measure, denote $|U|:=\int_U\,dx$. 
For a measurable subset $U\subset\R^d$, $p\ge 1$, let $L^p_{\loc}(U)$ be the set of all Lebesgue a.e. identified equivalence classes of measurable functions $f:U\to\R$ such that $f|_K\in L^p(K)$ for every compact set $K\subset U$.
For separable Hilbert spaces $V_1$, $V_2$, denote by $L(V_1,V_2)$ the space of linear operators from $V_1$ to $V_2$, and denote by $L_2(V_1,V_2)$ the space of \emph{Hilbert--Schmidt operators} from $V_1$ to $V_2$. Abbreviate $L(V_1):=L(V_1,V_1)$ and $L_2(V_1):=L_2(V_1,V_1)$.

\subsection*{Organization of the paper}

The assumptions and examples for our model are discussed in Section \ref{sec:2}. In Subsection \ref{subsec:well}, we discuss the main assumptions for well-posedness of equation \eqref{eq:gen_model}. In Subsection \ref{subsec:subspace}, we introduce the nonlocal Hilbert subspace, which is needed for the existence of invariant measures. In Subsection \ref{subsec:ergoassu}, we discuss the main assumptions for the existence and uniqueness of invariant measures. In Section \ref{sec:3}, we provide an existence and uniqueness result for solutions to \eqref{eq:gen_model}. The remaining Section \ref{sec:4} is devoted to the discussion of invariant measures. In Subsection \ref{subsec:4.1}, additional assumptions and the invariance of the nonlocal Hilbert space are discussed. In Subsection \ref{subsec:4.2}, we introduce the necessary background facts on invariant measures and prove the existence of at least one invariant measure. Finally, in Subsection \ref{subsec:4.3}, we introduce another assumption and prove the existence of at most one invariant measure, which is exponentially mixing and exponentially ergodic. We discuss the additional case of a monotone activation function in Subsection \ref{subsec:4.4}.

\section{Assumptions and examples}\label{sec:2}

Let $U\subset\R^d$ be the closure of a non-empty open domain. Let $\rho\in L^1_{\loc}(U)$ with $\rho> 0$ $dx$-a.e. on $U$, and denote $H:=L^2(U,\rho)$ with inner product
$\langle u,v\rangle:=\int_U u(x)v(x)\rho(x)\,dx$ and norm $\|u\|:=|\langle u,u\rangle|^{1/2}$.
We do not assume the connectivity kernel $w$ to be necessarily symmetric, as we aim to consider more general kernels covering cases where the neural interaction may occur only in one direction, or where one direction is strongly preferred. To this end, define 
\begin{align*}
(K u)(x):=\int_U w(x,y)u(y)\,dy,\quad\text{for a.e.}\; x\in U.
\end{align*}

\subsection{Assumptions for existence and uniqueness of solutions}\label{subsec:well}

We consider three separate cases for our setup of the spatial component in the domain $U$. The following definitions and results are needed for case (ii) of Assumption \ref{assu:space} below.
See \cite[Chapter 7]{Gra:14} for a discussion of Muckenhoupt weights.

\begin{definition}
The weight $\rho\in L^1_{\loc}(\R^d)$ is called an \emph{$A_2$-Muckenhoupt weight}, denoted $\rho\in A_2$, if
\begin{align*}
[\rho]_{A_2}:=\sup_{Q\subset\R^d,\;Q\;\textup{cube}}\left(\frac{1}{|Q|}\int_Q\rho(x)\, dx\right)\left(\frac{1}{|Q|}\int_Q   \rho(x)^{-1}\, dx\right)<\infty.
\end{align*}
\end{definition}

\begin{example}
Let
    \begin{align*}
        \rho(x) = |x|^{\theta},\quad x\in\R^d,\quad \theta\in(-d,d).
    \end{align*}
Then, $\rho\in A_2$, see \cite[Chapter IX, Proposition 3.2, Corollary 4.4]{Torchinsky1986}.
\end{example}

\begin{lemma}\label{lem:muckenhoupt_maximal}
If $\rho\in A_2$, then there exists a constant $C(d)>0$ such that for $u\in L^2(\R^d,\rho)$,
\begin{align*}
\|\mathfrak{M}(u)\|_{L^2(\R^d,\rho)}\leq C(d)[\rho]_{A_2}\|u\|_{L^2(\R^d,\rho)},
\end{align*}
where $\mathfrak{M}$ denotes the \emph{Hardy--Littlewood maximal} function given by
\begin{equation}\label{eq:Max}
\mathfrak{M}(u)(x): = \sup_{x\in B,\;B\;\textup{ball}}\frac{1}{|B|}\int_{B}|u(y)|\, dy,\quad x\in\R^d.
\end{equation}
\end{lemma}
\begin{proof}
See \cite[Theorem 7.1.9]{Gra:14}.
\end{proof}

The next definition can be found in \cite[Remark 2.1.11]{Gra:14}.

\begin{definition}
If for $J\in L^1(\R^d)$ there exists a non-negative decreasing function $j_0:[0,\infty)\to [0,\infty)$ that is continuous except at a finite number of points such that
\begin{align*}
|J(x)|\leq j_0(|x|) =:J_0(x),\quad x\in\R^d,
\end{align*}
then $J_0$ is called a \emph{radially decreasing majorant} of $J$.
\end{definition}

\begin{lemma}\label{lem:muckenhoupt_majorant}
If $J_0\in L^1(\R^d)$ is a radially decreasing majorant of $J\in L^1(\R^d)$,
then for all $u\in L^1_{\loc}(\R^d)$, we have the following pointwise estimate
\begin{align*}
|(u\ast J)(x)|\leq\|J_0\|_{L^1(\R^d)}\mathfrak{M}(u)(x),\quad x\in\R^d,
\end{align*}
where $u\ast J:=\int_{\R^d}J(\cdot-y) u(y)\,dy$ denotes \emph{convolution}, and $\Mfrak$ is as in \eqref{eq:Max}.
\end{lemma}
\begin{proof}
See \cite[Theorem 2.1.10, Remark 2.1.11, Corollary 2.1.12]{Gra:14}.
\end{proof}

We are ready to formulate the three cases of our assumptions on $U$, $f$ and $w$.

\begin{assumption}\label{assu:space}
Given that $H=L^2(U,\rho)$, we consider three different cases:
\begin{enumerate}
    \item $U$ is the closure of a non-empty, open, bounded subset of $\R^d$, while $f$ is Lipschitz continuous and $\rho\equiv 1$, and the kernel satisfies the integrability condition
\[\int_U\int_U |w(x,y)|^2\,dx\,dy<\infty.\]
    \item We consider the whole space $U=\R^d$ with Muckenhoupt weight $\rho\in A_2$.
 Assume that $f$ is Lipschitz continuous with $f(0)=0$, and that $w(x,y)=J(x-y)$ is a convolution kernel such that $J$ has a radially decreasing majorant $J_0\in L^1(\R^d)$.  
    \item $U$ is the closure of an open, non-empty subset of $\R^d$, $\rho> 0$ a.e., $\rho,\rho^{-1}\in L^1(U)$, $f$ is Lipschitz continuous, and the kernel satisfies the integrability and boundedness condition
\[\sup_{y\in U}\int_U |w(x,y)|^2\rho(x)\,dx<\infty.\]
\end{enumerate}
\end{assumption}
Note that in case (ii), $u\in L^2(\R^d,\rho)$ with $\rho\in A_2$ implies $u\in L^1_{\loc}(\R^d)$, by the Cauchy--Schwarz inequality and $\rho^{-1}\in L^1_{\loc}(\R^d)$, so that Lemma \ref{lem:muckenhoupt_majorant} applies.
Furthermore, in case (iii), the above assumption forces the Lebesgue measure of $U$ to be finite by the Cauchy--Schwarz inequality. Of course, $U$ may still be unbounded.

As the $\alpha$-decay term can be viewed as the linear ``principal part'' of the equation \eqref{eq:gen_model}, the interaction term can be viewed as a nonlinear perturbation. In the theory of semilinear stochastic partial differential equations (SPDEs), it is standard to assume Lipschitz continuity for the nonlinear term, see \cite{DaPrZa:2nd}, although many kinds of generalizations exist \cite{AV:24,LR:15,GessToe1}. The Lipschitz continuity of $f$ yields the Lipschitz continuity of the Nemytskii operator, see Lemma \ref{lem:nemytskii_lipschitz} below. In cases (i) and (iii), the integrability assumptions on $w$ ensure that $K$ is a compact Hilbert--Schmidt operator, which is a property needed in the construction of Subsection \ref{subsec:subspace} below; see Remark \ref{rem:caseii} for case (ii). The twofold assumptions on the weight $\rho$ are required from the viewpoint of the analysis in the proofs of Lemma \ref{lem:nemytskii_lipschitz} and the well-posedness Theorem~\ref{thm:existence_all_cases}.
We note that the first condition is weaker than the assumption on $w$ in \cite{KuehnRiedler}.
Furthermore, \cite{FaugerasInglis} already contains the case of unbounded domains with weights, however, our conditions are slightly more general, even if we have to assume the square integrability of $w$.

\begin{example}\label{ex:activationLip}
It is easy to verify that the following subset of Example \ref{ex:activation} is Lipschitz continuous.
\begin{description} 
\item[ReLU] $f(x)=x\vee 0$;
\item[Logistic] $f(x)=(1+e^{-x})^{-1}$;
\item[Hyperbolic tangent] $f(x)=\tanh(x)$.
\end{description}
\end{example}

The following assumption is standard in the theory of SPDEs with spatially degenerate noise, see \cite{DaPrZa:2nd,LR:15}.

\begin{assumption}\label{assu:noise}
    For the noise coefficient $B:L^2(U,\rho)\to L_2(V,L^2(U,\rho))$ for the model \eqref{eq:gen_model}, we assume that there exists a constant $C_B\ge 0$ with 
    \begin{align*}
        \lVert B(u)-B(v)\rVert_{L_2(V,H)}^2\leq C_{B}\lVert u-v\rVert^2,
    \end{align*}
    for every $u,v\in H$.
\end{assumption}
From this, we immediately get for $u\in H$,
\begin{align*}
\| B(u) \|_{L_2(V,H)} \leq \sqrt{C_B}\| u \| + \| B(0) \|_{L_2(V,H)}.
\end{align*}

Note that, in light of model \eqref{eq:specialmodel}, from the collection of Example \ref{ex:activationLip}, only the logistic function satisfies
that
\[\sqrt{f}:x\mapsto\sqrt{(1+e^{-x})^{-1}}\]
is Lipschitz continuous as well, as
\[(\sqrt{f})'(x)=\frac{1}{2}e^{-x}(1+e^{-x})^{-3/2}\le\frac{1}{3\sqrt{3}},\quad x\in\R,\]
where the global maximum is attained at $x=-\log(2)$.

We redefine the activation function $f$ in terms of a Nemytskii operator. Let 
\[F(u)(x):=f(u(x)),\quad u\in H\]
be the associated \emph{Nemytskii operator}. The properties of $F$ will depend on the assumptions we make on $f$, that is, the Lipschitz condition.
\begin{lemma}\label{lem:nemytskii_lipschitz}
Assume that $f$ is Lipschitz continuous and that $\rho\in L^1(U)$. Then the associated Nemytskii operator $F$ is a nonlinear, bounded, Lipschitz continuous operator on $H=L^2(U,\rho)$ with
\[\|F(u)\|\le \Lip(f)\|u\|+|f(0)|\|\rho\|_{L^1(U)}^{1/2},\]
and we have that $\Lip(F)=\Lip(f)$.
\end{lemma}
\begin{proof}
Given that $f\in\Lip(\R)$, we know that
\begin{align*}
    |f(x)-f(y)|\leq\Lip(f)|x-y|,\quad x,y\in\R,
\end{align*}
and therefore, we can also conclude that
\begin{align*}
    |f(x)| \leq\Lip(f)|x| + |f(0)|.
\end{align*}
Now, for $u\in H$, by Minkowski's inequality in $L^2(U,\rho)$,
\begin{align*}
\|F(u)\| &= \left(\int_U |f(u(x))|^2\rho(x)\, dx\right)^{1/2}\\
\leq& \left(\int_U \left(\Lip(f)|u(x)| + |f(0)|\right)^2\rho(x)\, dx\right)^{1/2}\\
\leq& \Lip(f)\left(\int_U |u(x)|^2\rho(x)\, dx\right)^{1/2} + |f(0)|\left(\int_U\rho(x)\,dx\right)^{1/2}\\
&= \Lip(f)\|u\| + |f(0)|\|\rho\|_{L^1(U)}^{1/2}.
\end{align*}
Similarly, for $u,v\in H$,
\begin{align*}
\|F(u)-F(v)\| &= \left(\int_U |f(u(x))-f(v(x))|^2\rho(x)\, dx\right)^{1/2}\\
\leq& \left(\int_U \Lip(f)^2|u(x)-v(x)|^2\rho(x)\, dx\right)^{1/2}\\
&= \Lip(f)\left(\int_U |u(x)-v(x)|^2\rho(x)\, dx\right)^{1/2}\\
&=\Lip(f)\|u-v\|.
\end{align*}
Hence $\Lip(F)\le\Lip(f)$. Conversely, choosing $A\subset U$ with $0<\nu(A)<\infty$ and $u=a\mathbbm{1}_A$, $v=b\mathbbm{1}_A$ with $a\neq b$ gives $\|F(u)-F(v)\|=|f(a)-f(b)|\nu(A)^{1/2}$ and $\|u-v\|=|a-b|\nu(A)^{1/2}$, so that $\Lip(F)\ge\Lip(f)$.
\end{proof}

If $f(0)=0$, we obtain the following result by inspecting the proof of Lemma \ref{lem:nemytskii_lipschitz}.

\begin{corollary}\label{cor:nemytskii_lipschitz}
Assume that $f$ is Lipschitz continuous and that $f(0)=0$. Then the associated Nemytskii operator $F$ is a nonlinear, bounded, Lipschitz continuous operator on $H=L^2(U,\rho)$ with
\[\|F(u)\|\le \Lip(f)\|u\|,\]
and we have that $\Lip(F)=\Lip(f)$.
\end{corollary}

By convention, to simplify the proofs below, if $f(0)=0$, we set $|f(0)|\|\rho\|^{1/2}_{L^1(U)}:=0$, even if $\rho\not\in L^1(U)$.

\subsection{The nonlocal Hilbert space}\label{subsec:subspace}

Let us discuss the assumptions and technical details related to the connectivity kernel $w$. The construction of the nonlocal Hilbert space is inspired by \cite{KuehnToelle19}.

Define
\[(K^\ast v)(x)=\frac{1}{\rho(x)}\int_U w(y,x)v(y)\rho(y)\,dy,\quad\text{for a.e.}\; x\in U.\]
By Fubini's theorem, it is straightforward to verify that $\langle Ku,v\rangle=\langle u,K^\ast v\rangle$ for every $u,v\in H$ and hence $K^\ast$ is the Hilbert space adjoint of $K$.
Define the symmetric and antisymmetric part of $K$, respectively, by
\[\hat{K}:=\frac{1}{2}(K+K^\ast),\quad \check{K}:=\frac{1}{2}(K-K^\ast),\]
so that $K=\hat{K}+\check{K}$, where $\hat{K}$ is self-adjoint and $\check{K}$ is skew-adjoint on $H$.
In case (i) of
Assumption \ref{assu:space}, where $\rho\equiv1$, $\hat{K}$ and $\check{K}$ are the integral operators associated to the kernels
\begin{align*}
    \hat{w}(x,y):=\frac{1}{2}[w(x,y)+w(y,x)],
\end{align*}
and
\begin{align*}
    \check{w}(x,y):=\frac{1}{2}[w(x,y)-w(y,x)],
\end{align*}
respectively.
If $w$ is symmetric, then $\hat{w}=w$.

Explicitly,
\begin{align*}
(\hat{K}u)(x)=\frac{1}{\rho(x)}\int_U \hat{w}_\rho(x,y)u(y)\,dy,\qquad
\hat{w}_\rho(x,y):=\frac{1}{2}[\rho(x)w(x,y)+\rho(y)w(y,x)],
\end{align*}
for a.e. $x\in U$, where the kernel $\hat{w}_\rho$ is symmetric, and analogously for
$\check{K}$ with the anti-symmetric kernel
$\check{w}_\rho(x,y):=\frac{1}{2}[\rho(x)w(x,y)-\rho(y)w(y,x)]$.

The next assumption is crucial to ensure the spectral properties needed for our construction of the nonlocal Hilbert space.

\begin{assumption}\label{assu:hat}
Suppose that
\[\int_U\int_U |w(x,y)|^2\frac{\rho(x)}{\rho(y)}\,dx\,dy<\infty.\]
    Furthermore, suppose that, for one and the same choice of sign throughout, for every $u\in L^2(U,\rho)$,
\[\pm\int_U\int_U w(x,y)u(x)u(y)\rho(x)\,dx\,dy\ge 0.\]
\end{assumption}

Since $\rho>0$ $dx$-a.e., the map $u\mapsto\rho^{1/2}u$ is unitary from $H$ onto
$L^2(U)$. Hence, if $\{\psi_k\}$ is a complete orthonormal system of $L^2(U)$,
then $\{\rho^{-1/2}\psi_k\}$ is one of $H$, and for a.e. $x\in U$,
\[(K\rho^{-1/2}\psi_k)(x)=\int_U w(x,y)\rho(y)^{-1/2}\psi_k(y)\,dy
=\big\langle w(x,\cdot)\rho^{-1/2},\psi_k\big\rangle_{L^2(U)}.\]
By Tonelli's theorem, the finiteness of the integral in Assumption \ref{assu:hat}
implies that\linebreak $w(x,\cdot)\rho^{-1/2}\in L^2(U)$ for a.e. $x\in U$, so that Parseval's
identity gives
\[\sum_{k}\big|(K\rho^{-1/2}\psi_k)(x)\big|^2=\int_U|w(x,y)|^2\rho(y)^{-1}\,dy .\]
Multiplying by $\rho(x)$, integrating, and interchanging sum and integral by monotone
convergence, we obtain
\[\|K\|^2_{L_2(H,H)}=\sum_{k}\big\|K\rho^{-1/2}\psi_k\big\|^2
=\int_U\int_U|w(x,y)|^2\frac{\rho(x)}{\rho(y)}\,dy\,dx<\infty,\]
so that $K\in L_2(H,H)$. As the Hilbert--Schmidt norm is invariant under taking adjoints,
also $\hat{K},\check{K}\in L_2(H,H)$, and thus $\hat{K}$ is a self-adjoint Hilbert--Schmidt
operator on $H$, compare with \cite{HS:78}.

Since $\check{K}$ is skew-adjoint, $\langle\check{K}u,u\rangle=0$ for every $u\in H$, and
hence the second condition of Assumption \ref{assu:hat} states equivalently that
$\langle\pm\hat{K}u,u\rangle\ge 0$ for every $u\in H$. In case (i) of Assumption \ref{assu:space}, the first condition of
Assumption \ref{assu:hat} is precisely the integrability condition assumed there, and in
case (iii) it follows from Assumption \ref{assu:space}, since
\begin{align*}&\int_U\int_U |w(x,y)|^2\frac{\rho(x)}{\rho(y)}\,dx\,dy
=\int_U\frac{1}{\rho(y)}\left(\int_U|w(x,y)|^2\rho(x)\,dx\right)dy\\
\le&\left(\sup_{y\in U}\int_U|w(x,y)|^2\rho(x)\,dx\right)\|\rho^{-1}\|_{L^1(U)}<\infty.\end{align*}

\begin{remark}\label{rem:caseii}
Case (ii) of Assumption \ref{assu:space} is incompatible with Assumption \ref{assu:hat}.
Indeed, for $U=\R^d$ and $w(x,y)=J(x-y)$, the substitution $x=y+z$ in the above
identity for the Hilbert--Schmidt norm gives
\[\|K\|^2_{L_2(H,H)}=\int_{\R^d}|J(z)|^2\left(\int_{\R^d}\frac{\rho(y+z)}{\rho(y)}\,dy\right)dz,\]
and for $\rho\in A_2$ the inner integral is infinite for every $z\in\R^d$. To see this, let $Q$ be a
unit cube, set $\rho(Q):=\int_Q\rho$ and $E:=\{y\in Q\;\colon\;\rho(y)\le 2\rho(Q)\}$, so that
$|E|\ge\frac{1}{2}$ by Chebyshev's inequality. Applying the $A_2$ condition on the cube
$Q^\ast\supset Q\cup (Q+z)$ of side length $1+2|z|$ concentric with $Q$, in the form
$\rho(F)\ge [\rho]_{A_2}^{-1}\rho(Q^\ast)(|F|/|Q^\ast|)^2$ for measurable $F\subset Q^\ast$,
we obtain
\[\int_Q\frac{\rho(y+z)}{\rho(y)}\,dy\ \ge\ \frac{\rho(E+z)}{2\rho(Q)}\ \ge\
\frac{1}{8[\rho]_{A_2}(1+2|z|)^{2d}},\]
which is independent of the position of $Q$, so that summing over a partition of $\R^d$ into unit
cubes yields the claim. Hence, unless $J=0$ $dx$-a.e., $K$ is not a Hilbert--Schmidt operator on
$H$. This reflects the translation invariance of the convolution kernel, which precludes the
compactness underlying the construction of Subsection \ref{subsec:subspace}; already for
$\rho\equiv 1$, the operator $K$ is a Fourier multiplier with symbol $\hat{J}$ and thus not
compact on $L^2(\R^d)$. For this reason, Section \ref{sec:4} is restricted to the cases (i) and
(iii), while case (ii) is covered by the well-posedness result
Theorem \ref{thm:existence_all_cases}.
\end{remark}

In terms of the kernel, in case (i) of Assumption \ref{assu:space} and for continuous $\hat{w}$, the second condition in Assumption \ref{assu:hat} is called \emph{Mercer's condition}, which is implied by
    \begin{equation}\label{eq:posneg}\pm\sum_{i,j=1}^N c_i c_j \hat{w}(x_i,x_j)\ge 0,\end{equation}
for any choice of $N\in\N$, $\{x_1,\ldots,x_N\}\subset U$, and $\{c_1,\ldots,c_N\}\subset\R$, see \cite{Ferreira:2013ei}.

We shall use the notation $\pm \hat{K}$ to denote $\hat{K}$ if $\hat{K}$ is non-negative semidefinite, and to denote $-\hat{K}$ if $\hat{K}$ is non-positive semidefinite, thus $\pm\hat{K}$ is non-negative semidefinite and self-adjoint.

Note that \eqref{eq:posneg} does neither imply, nor require that $\hat{w}\ge 0$ or $\hat{w}\le 0$ in a pointwise sense. In fact, the following examples from \cite[Section 4]{KuehnToelle19} provide examples of non-negative semidefinite and symmetric kernels for $U=\R^d$ and $\rho\equiv 1$.

\begin{example}[Non-negative semidefinite kernels]\label{ex:Fourier} The following kernels\footnote{The first three kernels are the characteristic functions of a \emph{centered Gaussian}, a \emph{centered Cauchy}, a \emph{centered Laplace distribution}, respectively. Kernel (iv) is the characteristic function of a \emph{uniform distribution} on $[-1,1]^d$. Kernel (v) is the characteristic function of a \emph{symmetric sum of Dirac distributions}.}, of the form $\hat{w}(x,y):= J(x-y)$ are non-negative semidefinite on $\R^d$ for (i)--(v), and on $\R$ for (vi)--(x).
\begin{enumerate}
\item $J(x)=e^{-\frac{1}{2}\langle x,Mx\rangle}$, where $M\in\R^{d\times d}$ is a symmetric and non-negative semidefinite matrix \emph{(centered Gaussian)};
\item $J(x)=e^{-\sqrt{\langle x,Mx\rangle}}$, where $M\in\R^{d\times d}$ is a symmetric and non-negative semidefinite matrix;
\item $J(x)=\left(1+\frac{1}{2}\langle x,Mx\rangle\right)^{-1}$, where $M\in\R^{d\times d}$ is a symmetric and non-negative semidefinite matrix;
\item $J(x)=J(x_1,\ldots,x_d)=\prod_{j=1}^d \frac{\sin(x_j)}{x_j}$, where the factors of the product are (by definition) equal to $1$ if $x_j=0$;
\item $J(x)=\sum_{i=1}^\infty a_i\cos(\langle m_i,x\rangle)$, where $a_i\ge 0$ with $\sum_{i=1}^\infty a_i=1$ and $m_i\in\R^d$, $m_i\not=\pm m_j$ for $i\not=j$;
\item $J(x)=(1-x^2)e^{-\frac{x^2}{2}}$ \emph{(Mexican hat)};
\item $J(x)=e^{-\frac{x^2}{2}}-Ae^{-\frac{x^2}{s^2}}$ for $\sqrt{2}\le s\le\frac{\sqrt{2}}{A}$ \emph{(another Mexican hat)};
\item $J(x)=e^{-\gamma_1|x|}-\Gamma e^{-\gamma_2|x|}$ for $0<\Gamma\le\frac{\gamma_2}{\gamma_1}$, and $\gamma_1>\gamma_2>0$ \emph{(yet another Mexican hat)};
\item $J(x)=e^{-b|x|}(b\sin(|x|)+\cos(|x|))$ for $b>0$;
\item $J(x)=\frac{1}{4}(1-|x|)e^{-|x|}$ \emph{(wizard hat)}.
\end{enumerate}
\end{example}

In order to access the above example in our situation, let $\pm\hat{w}:\R^d\times\R^d\to\R$ be non-negative semidefinite.
We note that in case (i) of Assumption \ref{assu:space}, for $u\in L^2(U)$,
\[\bar{u}(x):=\begin{cases}u(x)&\text{\; if\;} x\in U,\\ 0&\text{\; if\;}x\in\R^d\setminus U,\end{cases}\]
we have that $\bar{u}\in L^2(\R^d)$, and thus,
\[0\le\int_{\R^d}\int_{\R^d}(\pm\hat{w}(x,y))\bar{u}(x)\bar{u}(y)\,dx\,dy=\int_U\int_U(\pm\hat{w}(x,y))u(x)u(y)\,dx\,dy.\]
Hence $\pm\hat{w}:U\times U\to\R$ is non-negative semidefinite.

As Example \ref{ex:Fourier} relies on the Fourier transform and Bochner's theorem, in cases (ii) and (iii) of Assumption \ref{assu:space}, it is not clear if the kernels remain non-negative semidefinite in the weighted space.

The following construction of the nonlocal Hilbert space is taken from \cite{KuehnToelle19}, compare also with the projection method from \cite{A:23}.
Under Assumption \ref{assu:hat}, $\pm\hat{K}$ is a non-negative, self-adjoint,
Hilbert--Schmidt operator on $H$. By the spectral theorem for compact self-adjoint
operators, see \cite{ReSi1}, there exist a (possibly finite) orthonormal system $\{e_i\}$
in $H$ and positive eigenvalues $\{\lambda_i\}$ of $\pm\hat{K}$, arranged in
non-increasing order, such that
\[H=\operatorname{ker}(\pm\hat{K})\oplus\overline{\operatorname{span}\{e_i\}},\]
where the decomposition is orthogonal, and $\pm\hat{K}u=\sum_i\lambda_i\langle u,e_i\rangle e_i$
for every $u\in H$. If the system $\{e_i\}$ is infinite, then $\lambda_i\to 0$ as
$i\to\infty$, and $0$ belongs to the spectrum of $\pm\hat{K}$. Since $\pm\hat{K}$ is
Hilbert--Schmidt, we have $\{\lambda_i\}\in\ell^2$.

Denote by $H_0:=(\operatorname{ker}(\pm\hat{K}))^\perp=\overline{\operatorname{span}\{e_i\}}$
the closed subspace of $H$ on which $\pm\hat{K}$ is injective, and let
$(\pm\hat{K})^{-\frac{1}{2}}$ be the operator square root of the Moore--Penrose
pseudoinverse $(\pm\hat{K})^{-1}$ of $\pm\hat{K}$, see \cite[Section 2.1.2]{Hagen:vz};
equivalently, $(\pm\hat{K})^{-\frac{1}{2}}v=\sum_i\lambda_i^{-\frac{1}{2}}\langle v,e_i\rangle e_i$.
If $\{e_i\}$ is infinite, this operator is unbounded, and hence only densely defined on
$H_0$. We therefore define the \emph{nonlocal Hilbert space}
\[H_1:=(\pm\hat{K})^{\frac{1}{2}}(H_0)
=\Big\{v\in H_0\;\colon\;\sum_{i=1}^\infty\lambda_i^{-1}|\langle v,e_i\rangle|^2<\infty\Big\},\]
endowed with the (nonlocal) inner product and norm
\[\langle u,v\rangle_1:=\langle(\pm\hat{K})^{-\frac{1}{2}}u,(\pm\hat{K})^{-\frac{1}{2}}v\rangle,
\qquad\|v\|_1:=\|(\pm\hat{K})^{-\frac{1}{2}}v\|,\qquad u,v\in H_1.\]
By construction, $(\pm\hat{K})^{\frac{1}{2}}:H_0\to H_1$ is an isometric isomorphism, so
that $H_1$ is a separable Hilbert space; it is a dense linear subspace of $H_0$, but, in
contrast to $H_0$, it is not closed in $H$ whenever $\{e_i\}$ is infinite. Throughout, we
use the convention $\|v\|_1:=+\infty$ for $v\in H\setminus H_1$.
Note that
 \begin{equation}\label{eq:H_1-norm-bound}
        \|u\| \leq \|(\pm\hat{K})^{1/2}\|_{L(H)}\|u\|_1,\quad u\in H_1,
    \end{equation}
which follows from
\begin{align*}
\|u\| = \|(\pm\hat{K})^{1/2}(\pm\hat{K})^{-1/2} u\| \leq \|(\pm\hat{K})^{1/2}\|_{L(H)} \|(\pm \hat{K})^{-1/2}u\|,\quad u\in H_1.
\end{align*}

\begin{lemma}\label{lem:compact}
For any $R>0$,
\[\{u\in H\;\colon\;\|u\|_1\le R\}\]
is a compact subset of $H$.
\end{lemma}
\begin{proof}
Let $u_i\in\{u\in H\;\colon\;\|u\|_1\le R\}$, $i\in\N$ be any sequence. Then there exist $v_i\in H_0$ with $v_i= (\pm\hat{K})^{-1/2} u_i$ and $\|v_i\|\le R$, in other words, $\{v_i\}$ is contained in a bounded subset of $H$. However, by definition, $u_i=(\pm\hat{K})^{1/2}v_i$, and by the well-known fact that the operator square root of a compact operator is itself a compact operator (in our case, we can also use the spectral theorem), we get that a subsequence $\{u_{i_k}\}$ of
$\{u_i\}$ converges in $H$ to some $u_0\in H$. We can extract another non-relabeled subsequence, if necessary, such that $\{v_{i_k}\}$ converges weakly in $H$ to some $v_0$, which lies in $H_0$ since $H_0$ is weakly closed. By weak continuity of $(\pm\hat{K})^{1/2}$, $u_0=(\pm\hat{K})^{1/2}v_0$, and as the closed ball with radius $R$ in $H$ is also weakly closed, we also get that $\|v_0\|\le R$ and thus $\|u_0\|_1\le R$, and
hence $\{u\in H\;\colon\;\|u\|_1\le R\}$ is sequentially compact in $H$, and thus compact, $H$ being a metric space.
\end{proof}

\subsection{Assumptions for existence and uniqueness of invariant measures}\label{subsec:ergoassu}

The existence of invariant measures proved in Theorem \ref{thm:existence} requires the following additional assumption.

\begin{assumption}\label{assu:inv}
Suppose that there exist constants $\tilde{C}_B,\tilde{\tilde{C}}_B\ge 0$ such that
\begin{align*}
\|B(u)\|_{L_2(V,H_1)}^2 \leq \tilde{C}_B\|u\|_1^2 + \tilde{\tilde{C}}_{B},\quad u\in H_1,
\end{align*}
that there exists a constant $\hat{C}_B\ge 0$ such that
\begin{align*}
\|B(u)-B(v)\|_{L_2(V,H_1)}\leq \hat{C}_B\|u-v\|_1,\quad u,v\in H_1,
\end{align*}
and that there exists a constant $C_{\check{K}}\ge 0$ such that
\begin{align*}
\|(\pm\hat{K})^{-1}\check{K}v\|\leq C_{\check{K}}\|v\|,\quad v\in H.
\end{align*}
Set
\[\beta:=2(1+C_{\check{K}})\|(\pm\hat{K})^{\frac{1}{2}}\|_{L(H)}^2 \left(\Lip(f) + |f(0)|\|\rho\|_{L^1(U)}^{1/2}\right)+\tilde{C}_B,\]
and assume that
\[2\alpha-\beta>0.\]
\end{assumption}

Note that $\tilde{C}_B=\hat{C}_B=0$ for additive noise.
The second condition is the counterpart in $H_1$ of Assumption \ref{assu:noise}; it implies the first one with $\tilde{C}_B=2\hat{C}_B^2$ and $\tilde{\tilde{C}}_B=2\|B(0)\|^2_{L_2(V,H_1)}$, which is kept separately only because it may hold with smaller constants. Together with the bound
\[\|KF(u)-KF(v)\|_1\le (1+C_{\check{K}})\|(\pm\hat{K})^{\frac{1}{2}}\|_{L(H)}\Lip(f)\|u-v\|,\quad u,v\in H,\]
it ensures that \eqref{eq:gen_model} is well posed in $H_1$: the conditions of \cite[Theorem 7.5]{DaPrZa:2nd} are verified in $H_1$ exactly as in the proof of Theorem \ref{thm:existence_all_cases} below, since $e^{-\alpha t}\operatorname{Id}$ is a contraction semigroup on $H_1$ as well, and the resulting $H_1$-valued solution coincides with the $H$-valued one by uniqueness in $H$ together with the continuous embedding $H_1\hookrightarrow H$ from \eqref{eq:H_1-norm-bound}. In particular, for $u_0\in H_1$, the solution is a continuous $H_1$-valued semimartingale. This justifies the application of It\^o's formula to $\|\cdot\|_1^2$ and the use of the stopping times defined through $\|\cdot\|_1$ in Subsections \ref{subsec:4.1} and \ref{subsec:4.4} below.
The third condition can be interpreted as the domination of the anti-symmetric part $\check{K}$ of $K$ by the symmetric part $\hat{K}$ of $K$.
Note that $C_{\check{K}}=0$ if $w$ is symmetric.
The fourth condition ensures that the decay term dominates the right-hand side in the a priori estimates. The obtained bound enables us to obtain tightness of the time-averaged transition probabilities with initial data in $H_1$, see Proposition \ref{prop:tight} below.

The uniqueness, exponential ergodicity and exponential mixing of invariant measures proved in Theorem \ref{thm:uniqueness} require the following further assumption.

\begin{assumption}\label{assu:ergo}
Assume that
\[\lambda:=2\|K\|_{L(H)}\Lip(f)+C_B<2\alpha,\]
where $C_B\ge 0$ is as in Assumption \ref{assu:noise}.
\end{assumption}

This assumption ensures that the drift contribution in the It\^o formula for the exponentially weighted squared norm $e^{(2\alpha-\lambda)s}\|\cdot\|^2$ is non-positive, which establishes the exponential decay of the squared norm of two solutions with different initial data, that is, the synchronous coupling. As already mentioned in the introduction, this condition resembles the assumptions for the exponential frame in \cite{DLL:20,ESSTvG,BDP} and displays the natural hierarchy of parameters of the model.

\section{Existence and uniqueness of solutions}\label{sec:3}
We can now move on to establishing existence and uniqueness of strong solutions to \eqref{eq:gen_model}. Let $B$ be a map from $H$ to $L_2(V,H)$, the space of linear Hilbert--Schmidt operators from $V$ to $H$. We introduce the usual notions of strong, weak and mild solutions.
Let $S(t):H\to H$ be defined by $u\mapsto e^{-\alpha t} u$. $\{S(t)\}_{t\ge 0}$ is a $C_0$-semigroup on $H$ with infinitesimal generator\footnote{Note that the domain of the infinitesimal generator of $\{S(t)\}_{t\ge 0}$ is all of $H$.} $u\mapsto -\alpha u$. Set $KF:H\to H$, $KF(u):=(K\circ F)(u)$, where $F$ is the Nemytskii operator on $H$ associated to the nonlinear activation function $f:\R\to\R$.

\begin{definition}\label{def:sol}\phantom{ }
\begin{enumerate}[(I)]
\item  A strong solution $u(t)$, $t\in[0,T]$ to \eqref{eq:gen_model} is an $H$-valued predictable process, such that
\begin{align*}
&\mathbb{P}\left(\int_0^T (\|u(s)\| + \|-\alpha u(s)+KF(u(s))\|)\, ds <\infty \right) = 1,\\
&\mathbb{P}\left(\int_0^T \|B(u(s))\|^2_{L_2(V,H)}\, ds<\infty\right) = 1,
\end{align*}
and satisfies for arbitrary $t\in[0,T]$,
\begin{align*}
&u(t) = u(0) + \int_0^t\left(-\alpha u(s) + KF(u(s))\right)\, ds + \int_0^t B(u(s))\, dW(s),
\end{align*}
$\mathbb{P}$-a.s. in $H$.

\item A weak solution $u(t)$, $t\in[0,T]$ to \eqref{eq:gen_model} is an $H$-valued predictable process, such that
\begin{align*}
&\mathbb{P}\left(\int_0^T\|u(s)\|\, ds<\infty \right) = 1,\quad
\mathbb{P}\left(\int_0^T\|B(u(s))\|^2_{L_2(V,H)}\, ds<\infty\right) = 1,
\end{align*}
and satisfies for arbitrary $t\in[0,T]$,
\begin{align*}
\langle u(t),v\rangle
=&\langle u(0),v\rangle + \int_0^t\left(\langle u(s),-\alpha v\rangle + \langle KF(u(s)),v \rangle\right)\, ds
+ \int_0^t\langle v,B(u(s))\, dW(s)\rangle,
\end{align*}
$\mathbb{P}$-a.s. for every $v\in H$.

\item A mild solution $u(t)$, $t\in[0,T]$ to \eqref{eq:gen_model} is an $H$-valued predictable process that satisfies 
\begin{align*}
&\mathbb{P}\left(\int_0^T \|u(s)\|\, ds<\infty  \right) = 1,\quad\mathbb{P}\left(\int_0^T \|B(u(s))\|^2_{L_2(V,H)}\, ds<\infty\right) = 1,
\end{align*}
and satisfies for arbitrary $t\in[0,T]$,
\begin{align*}
&u(t) = S(t)u(0) + \int_0^t S(t-s)KF(u(s))\, ds+\int_0^t S(t-s)B(u(s))\, dW(s),
\end{align*}
$\mathbb{P}$-a.s. in $H$.
\end{enumerate}
\end{definition}

\begin{remark}
A strong solution is automatically a weak solution. Under our assumptions, a mild solution is a strong solution, and vice versa, see \cite[Theorem 6.5]{DaPrZa:2nd}. See also \cite[Proposition G.0.5]{LR:15}. Note that our notion of strong and weak solutions, respectively,
coincides with the notion of \emph{analytically strong} and \emph{analytically weak solutions} in \cite[Appendix G]{LR:15}, respectively.
All of our solutions are strong solutions in the probabilistic sense, that is, a unique solution exists for any stochastic basis
$(\Omega,\allowbreak\mathcal{F},\allowbreak\{\mathcal{F}_t\}_{t\ge 0},\allowbreak\P,\allowbreak\{W(t)\}_{t\ge 0})$.
\end{remark}

The following theorem is verified with the help of a standard result \cite[Theorem 7.5]{DaPrZa:2nd}. We would like to point out that rigorous results on existence and uniqueness of solutions to \eqref{eq:gen_model} have been previously obtained in \cite{KuehnRiedler,FaugerasInglis} for symmetric kernels. Note that existence and uniqueness of solutions to \eqref{eq:gen_model} do not rely on the symmetry of the kernel $w$, and that our assumptions slightly differ from those in \cite{FaugerasInglis}. We prove Theorem \ref{thm:main} (i).

\begin{theorem}\label{thm:existence_all_cases}
Assume that Assumptions \ref{assu:space} and \ref{assu:noise} are satisfied. Then, for any $p\ge 2$, any $T>0$, and for any initial datum $u(0)\in L^p(\Omega,\mathcal{F}_0,\P;H)$, there exists a unique $\P$-a.s. continuous and $\{\mathcal{F}_t\}_{t\ge 0}$-adapted $H$-valued
strong solution $u(t)$, $t\in [0,T]$ of \eqref{eq:gen_model}, and there exists a constant $C_{p,T}>0$, such that
\begin{align*}
\mathbb{E}(\|u(t)\|^p)\leq C_{p,T}\left(1+\mathbb{E}\|u(0)\|^p\right),\quad t\in[0,T].
\end{align*} 
\end{theorem}
\begin{proof}
We first verify (a) separately in each of the three cases of Assumption \ref{assu:space}, and then (b) and (c). The main argument is to show that the conditions listed in \cite[Theorem 7.5]{DaPrZa:2nd} are satisfied by the drift and noise operators. More precisely, it must be shown that
\begin{enumerate}[(a)]
\item There exists a constant $C>0$ such that
\[\|K F(u)-KF(v)\|\le C\|u-v\|\]
for every $u,v\in H$ and
\[\|KF(u)\|\le C(1+\|u\|)\]
for every $u\in H$.
\item There exists $s_0\in L^2_{\loc}([0,\infty))$ such that
\[\|S(t)B(u)\|_{L_2(V,H)}\le s_0(t)(1+\|u\|)\]
for every $t>0$ and every $u\in H$, and
\[\|S(t)B(u)-S(t)B(v)\|_{L_2(V,H)}\le s_0(t)\|u-v\|\]
for every $t>0$ and every $u,v\in H$.
\item There exists $r_0\in (0,\frac{1}{2})$ such that
\[\int_0^1 t^{-2r_0}s_0^2(t)\,dt<\infty.\]
\end{enumerate}
\textbf{Proof of (a):}\\
\textbf{Case (i):}\\
Recall that by assumption,
\[\kappa:=\int_U\int_U |w(x,y)|^2\,dx\,dy<\infty.\]
Let $u\in H$. By the Cauchy--Schwarz inequality in $L^2(U)$, 
\begin{align*}
\|Ku\|^2 &= \int_U\left(\int_U w(x,y)u(y)\, dy\right)^2\,dx\\
\leq&\int_U\left(\int_U |w(x,y)|^2\,dy\right)\left(\int_U |u(y)|^2\,dy\right)\,dx\\
\leq&\kappa\|u\|^2.
\end{align*}
As a consequence,
\begin{align*}
\|KF(u)\| \leq \sqrt{\kappa}\|F(u)\|,\quad u\in H,
\end{align*}
and
\begin{align*}
\|KF(u)-KF(v)\| &\le \sqrt{\kappa}\|F(u)-F(v)\|,\quad u,v\in H.
\end{align*}
The proof of (a) in the case (i) is concluded by Lemma \ref{lem:nemytskii_lipschitz}.\\
\textbf{Case (ii):}\\
Here $U=\R^d$, $\rho\in A_2$ and $f(0)=0$. In case (ii) assume that the neural kernel is given in the form of a classical convolution, that is,
\begin{align*}
w(x,y) = J(x-y),\quad x,y\in\R^d.
\end{align*}
In general, $J(x)\neq J(-x)$, but $J$ is assumed to have a radially decreasing majorant $J_0\in L^1(\R^d)$. Observe that
\begin{align*}
Ku = u\ast J= \int_{\R^d}J(\cdot-y)u(y)\, dy. 
\end{align*}
Now, for $u\in H$, we combine Lemmas \ref{lem:muckenhoupt_maximal} and \ref{lem:muckenhoupt_majorant} to obtain
\begin{align*}
\|u\ast J\|& \leq \|J_0\|_{L^1(\R^d)} \|\mathfrak{M}(u)\|\leq C(d)[\rho]_{A_2}\|J_0\|_{L^1(\R^d)}\|u\|.
\end{align*}
Denote
\begin{align*}
K_{\rho}:= C(d)[\rho]_{A_2}\|J_0\|_{L^1(\R^d)}.
\end{align*}
As a consequence,
\begin{align*}
\|KF(u)\| \leq K_{\rho}\|F(u)\|,\quad u\in H,
\end{align*}
and
\begin{align*}
\|KF(u)-KF(v)\| &\le K_{\rho}\|F(u)-F(v)\|,\quad u,v\in H.
\end{align*}
The proof of (a) in the case (ii) is concluded by Corollary \ref{cor:nemytskii_lipschitz}.\\
\textbf{Case (iii):}\\
Recall that by assumption,
\[\Lambda:=\sup_{y\in U}\int_U |w(x,y)|^2\rho(x)\,dx<\infty.\]
Let $u\in H$. By the Cauchy--Schwarz inequality in $L^2(U)$,
\begin{align*}
\|Ku\|^2 &= \int_U\left|\int_U w(x,y)u(y)\, dy\right|^2\rho(x)\, dx\\
\leq& \int_U\left(\int_U|w(x,y)||u(y)|\frac{\rho(y)}{\rho(y)}\, dy\right)^2\rho(x)\,dx\\
\leq& \int_U\left(\int_U|w(x,y)|^2|u(y)|^2\rho(y)\, dy\right)\left(\int_U\frac{1}{\rho(y)}\,dy\right)\rho(x)\,dx\\
\leq& \Lambda\|u\|^2\|\rho^{-1}\|_{L^1(U)}.
\end{align*}
Let
\[K_{\Lambda,\rho}:=\sqrt{\Lambda\|\rho^{-1}\|_{L^1(U)}}.\]
As a consequence,
\begin{align*}
\|KF(u)\| \leq K_{\Lambda,\rho}\|F(u)\|,
\end{align*}
and
\begin{align*}
\|KF(u)-KF(v)\| &\le K_{\Lambda,\rho}\|F(u)-F(v)\|,
\end{align*}
for every $u,v\in H$. The proof of (a) in the case (iii) is concluded by Lemma \ref{lem:nemytskii_lipschitz}.\\
\textbf{Proof of (b) and (c):}\\
It follows directly from Assumption \ref{assu:noise} that (b) is satisfied with $s_0(t):=Me^{-\alpha t}$, where $M:=\sqrt{C_B}\vee\|B(0)\|$. For $r_0:=\frac{1}{4}$, we obtain that
\[\int_0^1 t^{-2r_0}s_0^2(t)\,dt\le M^2\int_0^1 t^{-\frac{1}{2}}\,dt=2M^2<\infty.\]
Hence (c) is satisfied.
\end{proof}

\section{Existence and uniqueness of invariant measures}\label{sec:4}
In the present section, we consider either case (i) or case (iii) from Assumption \ref{assu:space}, see Remark \ref{rem:caseii}.
Now, we shall employ the subspace $H_1$ defined in Subsection \ref{subsec:subspace}.

\subsection{An invariant subspace}\label{subsec:4.1}

Before we prove the existence of an invariant measure, we present a result showing that $H_1$ is invariant for the solution. 

\begin{proposition}[An invariance result]\label{prop:inv}
Suppose that
Assumptions \ref{assu:space}, \ref{assu:noise}, \ref{assu:hat}, and \ref{assu:inv} hold.
Let $\beta\ge 0$ be as in Assumption \ref{assu:inv}.
Then, for any $\delta\in (0,1)$, for
\[\eta:=2(1+C_{\check{K}})|f(0)|\|\rho\|_{L^1(U)}^{1/2}+\left(1+\frac{9}{\delta}\right)\tilde{\tilde{C}}_B,\]
and for any $u_0\in H_1$, and any $t\in [0,T]$, it holds that
\begin{align*}
(1-\delta)\mathbb{E}\left(\sup_{0\leq s\leq t}\|u(s)\|^2_1\right) \leq \|u_0\|^2_1 + \left(\beta+\frac{9}{\delta}\tilde{C}_B\right)\E\int_0^t\|u(s)\|_1^2\,ds + \eta t.
\end{align*}
Furthermore, for
\[\tilde{\tilde{\eta}}:=2(1+C_{\check{K}})|f(0)|\|\rho\|_{L^1(U)}^{1/2}+\tilde{\tilde{C}}_B,\]
and for any $u_0\in H_1$, and any $t\in [0,T]$, it holds that
\begin{align*}
\mathbb{E}\left(\|u(t)\|^2_1\right)+(2\alpha-\beta)\E\int_0^t\|u(s)\|_1^2\,ds \leq \|u_0\|^2_1 + \tilde{\tilde{\eta}} t.
\end{align*}
In particular, for $u_0\in H_1$, $u\in L^\infty([0,T];L^2(\Omega;H_1))\cap L^2([0,T]\times\Omega;H_1)$. 
\end{proposition}
\begin{proof}
Applying the It\^o formula to the functional $u\mapsto\|u\|^2_1$ yields
\begin{align*}
\|u(t)\|_1^2 = \|u_0\|_1^2 &+ 2\int_0^t\langle -\alpha u(s) + KF(u(s)),u(s)\rangle_1\, ds\\
&+2\int_0^t\langle u(s),B(u(s))\, dW(s)\rangle_1\\
&+\int_0^t\operatorname{Tr}_{H_1}(B(u(s))(B(u(s)))^*)\, ds.
\end{align*}
Taking the supremum over finite time $[0,T]$ and taking expectations gives
\begin{align*}
\mathbb{E}\left(\sup_{0\leq s\leq T}\|u(s)\|_1^2\right)\leq&\|u_0\|_1^2
+ 2\mathbb{E}\left(\sup_{0\leq t\leq T}\int_0^t\langle -\alpha u(s)+KF(u(s)),u(s)\rangle_1\, ds\right)\\
&+ 2\mathbb{E}\left(\sup_{0\leq t\leq T}\left|\int_0^t\langle u(s),B(u(s))\, dW(s)\rangle_1\right|\right)\\
&+ \mathbb{E}\left(\sup_{0\leq t\leq T}\int_0^t\operatorname{Tr}_{H_1}(B(u(s))(B(u(s)))^*)\, ds\right).
\end{align*}
We handle each term separately. Define a sequence of stopping times 
\begin{align*}
\tau_N:=\inf\{s\in[0,T]:\|u(s)\|_1^2>N\}\wedge T,\quad N\in\N.
\end{align*}
Let $\{\hat{e}_i\}$ be a complete orthonormal system of $H_1$. First, observe that
\begin{align*}
\operatorname{Tr}_{H_1}(B(u(s))(B(u(s)))^*) &= \sum_{j=1}^\infty\langle B(u(s))(B(u(s)))^*\hat{e}_j,\hat{e}_j\rangle_1,
\end{align*}
where the adjoint is taken in $H_1$.
Therefore, we may conclude by Parseval's identity in $H_1$ and in $V$, using the assumption on $B$,
\begin{align*}
\operatorname{Tr}_{H_1}(B(u(s))(B(u(s)))^*) = \|B(u(s))\|^2_{L_2(V,H_1)}\leq \tilde{C}_B\|u(s)\|^2_1 + \tilde{\tilde{C}}_B.
\end{align*}
Thus,
\begin{align*}
&\mathbb{E}\left(\sup_{0\leq r\leq t\wedge\tau_N}\int_0^r\operatorname{Tr}_{H_1}(B(u(s))(B(u(s)))^*)\, ds\right)\leq \mathbb{E}\left(\sup_{0\leq r\leq t\wedge\tau_N}\int_0^r\|B(u(s))\|_{L_2(V,H_1)}^2\, ds\right)\\
\leq& \tilde{C}_B\mathbb{E}\left(\sup_{0\leq r\leq t\wedge\tau_N}\int_0^r\|u(s)\|_1^2\, ds\right) + \tilde{\tilde{C}}_B(t\wedge\tau_N).
\end{align*}
Next, we bound the stochastic integral term by means of a Burkholder--Davis--Gundy type inequality for $p=1$, see \cite{I:86}. Let $\eps>0$.
\begin{align*}
&2\mathbb{E}\left(\sup_{0\leq r\leq t\wedge\tau_N}\left|\int_0^{r}\langle u(s),B(u(s))\, dW(s)\rangle_1\right|\right) \\
\leq& 6\mathbb{E}\left[\left(\int_0^{t\wedge\tau_N}\|u(s)\|_1^2 \|B(u(s))\|_{L_2(V,H_1)}^2\, ds \right)^{1/2}\right]\\
\leq& 6\mathbb{E}\left[\left(\sup_{0\leq s\leq t\wedge\tau_N}\|u(s)\|_1^2\int_0^{t\wedge\tau_N}\|B(u(s))\|^2_{L_2(V,H_1)}\, ds\right)^{1/2}\right]\\
\leq& 6\mathbb{E}\left[\left(\sup_{0\leq s\leq t\wedge\tau_N}\|u(s)\|_1^2\int_0^{t\wedge\tau_N}(\tilde{C}_B\|u(s)\|_1^2 + \tilde{\tilde{C}}_B)\, ds\right)^{1/2}\right]\\
\leq& 6\mathbb{E}\left[\left(\sup_{0\leq s\leq t\wedge\tau_N}\varepsilon\|u(s)\|_1^2\frac{1}{\varepsilon}\int_0^{t\wedge\tau_N}(\tilde{C}_B\|u(s)\|_1^2 + \tilde{\tilde{C}}_B)\, ds\right)^{1/2}\right]\\
\leq& 3\mathbb{E}\left(\varepsilon\sup_{0\leq s\leq t\wedge\tau_N}\|u(s)\|_1^2\right)
 + 3\mathbb{E}\left(\frac{1}{\varepsilon}\int_0^{t\wedge\tau_N}(\tilde{C}_B\|u(s)\|_1^2 + \tilde{\tilde{C}}_B)\, ds\right)\\
&= 3\varepsilon\mathbb{E}\left(\sup_{0\leq s\leq t\wedge\tau_N}\|u(s)\|_1^2\right) + \frac{3\tilde{C}_B}{\varepsilon}\mathbb{E}\int_0^{t\wedge\tau_N}\|u(s)\|_1^2\, ds + \frac{3\tilde{\tilde{C}}_B(t\wedge\tau_N)}{\varepsilon},
\end{align*}
where we have used Young's inequality.

Let us deal with the term involving the operator $K$,
\begin{align*}
&2\mathbb{E}\left(\sup_{0\leq r\leq t\wedge\tau_N}\int_0^{r}\langle -\alpha u(s) + KF(u(s)),u(s)\rangle_1\, ds\right)\\
&\le 2\mathbb{E}\left(\int_0^{t\wedge\tau_N}\left|\langle KF(u(s)),u(s)\rangle_1\right|\, ds\right),
\end{align*}
 where we have used that $-\alpha\|u(s)\|_1^2\le 0$ and that $\sup_{0\le r\le t\wedge\tau_N}\int_0^r g(s)\,ds\le\int_0^{t\wedge\tau_N}|g(s)|\,ds$ for integrable $g$.
The remaining term can be written  
\begin{align*}
&2\mathbb{E}\left(\int_0^{t\wedge\tau_N}\left|\left\langle (\hat{K}+\check{K})F(u(s)),u(s)\right\rangle_1\right|\, ds\right)\\
&= 2\mathbb{E}\left(\int_0^{t\wedge\tau_N}\left|\langle (\pm\hat{K})^{-1}(\hat{K}+\check{K})F(u(s)),u(s)\rangle\right|\, ds\right)\\
&\leq 2\mathbb{E}\left(\int_0^{t\wedge\tau_N}[|\langle\pm F(u(s)),u(s)\rangle| + |\langle (\pm\hat{K})^{-1}\check{K}F(u(s)),u(s)\rangle|]\, ds\right).
\end{align*}
The first term inside the integral can be bounded using the Lipschitz assumption on $f$ together with \eqref{eq:H_1-norm-bound}, Lemma \ref{lem:nemytskii_lipschitz}, and the inequality $x\leq x^2+1$:
\begin{align*}
&|\langle \pm F(u(s)),u(s) \rangle| = |\langle F(u(s)),u(s) \rangle|\leq \|F(u(s))\|  \|u(s)\|\\
\leq& \|(\pm\hat{K})^{\frac{1}{2}}\|_{L(H)} \|F(u(s))\| \|u(s)\|_1\\
\leq& \|(\pm\hat{K})^{\frac{1}{2}}\|_{L(H)} \|u(s)\|_1\left(\Lip(f)\|u(s)\| + |f(0)|\|\rho\|_{L^1(U)}^{1/2}\right)\\
\leq& \|(\pm\hat{K})^{\frac{1}{2}}\|^2_{L(H)} \|u(s)\|_1^2\Lip(f) + |f(0)|\|\rho\|_{L^1(U)}^{1/2}\left(\|(\pm\hat{K})^{\frac{1}{2}}\|^2_{L(H)} \|u(s)\|_1^2+1 \right)\\
=& \|(\pm\hat{K})^{\frac{1}{2}}\|_{L(H)}^2 \|u(s)\|_1^2\left(\Lip(f) + |f(0)|\|\rho\|_{L^1(U)}^{1/2}\right)+|f(0)|\|\rho\|_{L^1(U)}^{1/2}.
\end{align*}
In other words,
\begin{align*}
&2\mathbb{E}\int_0^{t\wedge\tau_N}|\langle \pm F(u(s)),u(s)\rangle|\, ds\\
\leq&  2\|(\pm\hat{K})^{\frac{1}{2}}\|_{L(H)}^2 \left(\Lip(f) + |f(0)|\|\rho\|_{L^1(U)}^{1/2}\right)\mathbb{E}\int_0^{t\wedge\tau_N}\|u(s)\|_1^2\, ds\\
&+2|f(0)|\|\rho\|_{L^1(U)}^{1/2}(t\wedge\tau_N).
\end{align*}
Let us handle the second integral term, where we use the assumption on $\check{K}$, together with \eqref{eq:H_1-norm-bound}, Lemma \ref{lem:nemytskii_lipschitz},
\begin{align*}
&|\langle (\pm\hat{K})^{-1}\check{K}F(u(s)),u(s)\rangle|\\
\leq& \|(\pm\hat{K})^{-1}\check{K}F(u(s))\| \|u(s)\|\\
\leq& C_{\check{K}}\|(\pm\hat{K})^{\frac{1}{2}}\|_{L(H)}\|F(u(s))\| \|u(s)\|_1\\
\leq& C_{\check{K}}\|(\pm\hat{K})^{\frac{1}{2}}\|_{L(H)} \|u(s)\|_1\left(\Lip(f)\|u(s)\| + |f(0)|\|\rho\|_{L^1(U)}^{1/2}\right)\\
\leq& C_{\check{K}}\|(\pm\hat{K})^{\frac{1}{2}}\|^2_{L(H)} \|u(s)\|_1^2\Lip(f) + C_{\check{K}}|f(0)|\|\rho\|_{L^1(U)}^{1/2}\left(\|(\pm\hat{K})^{\frac{1}{2}}\|^2_{L(H)} \|u(s)\|_1^2+1 \right)\\
=& C_{\check{K}}\|(\pm\hat{K})^{\frac{1}{2}}\|_{L(H)}^2 \|u(s)\|_1^2\left(\Lip(f) + |f(0)|\|\rho\|_{L^1(U)}^{1/2}\right)+C_{\check{K}}|f(0)|\|\rho\|_{L^1(U)}^{1/2}.
\end{align*}
We have obtained
\begin{align*}
&2\mathbb{E}\int_0^{t\wedge\tau_N}\left|\langle KF(u(s)),u(s)\rangle_1\right|\, ds\\
\leq&  2(1+C_{\check{K}})\|(\pm\hat{K})^{\frac{1}{2}}\|_{L(H)}^2 \left(\Lip(f) + |f(0)|\|\rho\|_{L^1(U)}^{1/2}\right)\mathbb{E}\int_0^{t\wedge\tau_N}\|u(s)\|_1^2\, ds\\
&+2(1+C_{\check{K}})|f(0)|\|\rho\|_{L^1(U)}^{1/2}(t\wedge\tau_N).
\end{align*}
For simplicity, set
\[\tilde{\eta}:=2(1+C_{\check{K}})|f(0)|\|\rho\|_{L^1(U)}^{1/2},\]
and recall $\beta$ from Assumption \ref{assu:inv}.
Collecting all the terms gives
\begin{align*}
&(1-3\varepsilon)\mathbb{E}\left(\sup_{0\leq s\leq t\wedge\tau_N}\|u(s)\|_1^2\right)\\
\leq& \|u_0\|_1^2
+ \left(\beta+\frac{3\tilde{C}_B}{\eps}\right) \mathbb{E}\int_0^{t\wedge\tau_N}\|u(s)\|_1^2\, ds
 + \left[\tilde{\tilde{C}}_B + \frac{3\tilde{\tilde{C}}_B}{\varepsilon} + \tilde{\eta}\right](t\wedge\tau_N).
\end{align*}
We obtain for $\eps:=\frac{\delta}{3}$, for $\delta\in (0,1)$,
\begin{align*}
&(1-\delta)\mathbb{E}\left(\sup_{0\leq s\leq t\wedge\tau_N}\|u(s)\|_1^2\right)\\
\leq& \|u_0\|_1^2
+ \left(\beta+\frac{9}{\delta}\tilde{C}_B\right) \mathbb{E}\int_0^{t\wedge\tau_N}\|u(s)\|_1^2\, ds
 + \left[\left(1+\frac{9}{\delta}\right)\tilde{\tilde{C}}_B + \tilde{\eta}\right](t\wedge\tau_N).
\end{align*}
By the convergence of the stopping times
\begin{align*}
\lim_{N\rightarrow\infty}\tau_N = T,\quad \mathbb{P}\text{-a.s.},
\end{align*}
by the monotone convergence theorem and by Fatou's lemma,
\begin{align*}
&(1-\delta)\mathbb{E}\left(\sup_{0\leq s\leq t}\|u(s)\|_1^2\right)\\
\leq& \|u_0\|_1^2
+ \left(\beta+\frac{9}{\delta}\tilde{C}_B\right) \mathbb{E}\int_0^t\|u(s)\|_1^2\, ds
 + \left[\left(1+\frac{9}{\delta}\right)\tilde{\tilde{C}}_B + \tilde{\eta}\right]t.
\end{align*}
The second claim can be proved analogously to the above, but taking the expectation before the application of Burkholder--Davis--Gundy inequality, and using the fact that under our assumptions on $B$,
\[M_t:=\int_0^{t}\langle u(s),B(u(s))\, dW(s)\rangle_1\]
is a local martingale, and hence for any $N\in\N$,
\begin{align*}
\mathbb{E}\left[M_{t\wedge\tau_N}\right] = \mathbb{E}\left[M_0\right] = 0.
\end{align*}
\end{proof}

\begin{corollary}
Under the assumptions of Proposition \ref{prop:inv}, we obtain that for any $\delta\in(0,1)$ and any $u_0\in H_1$,
\begin{align*}
(1-\delta)\mathbb{E}\left(\sup_{0\leq s\leq t}\|u(s)\|^2_1\right)\leq C_\delta\left(\|u_0\|^2_1+\eta t\right),\quad t\in[0,T],
\end{align*}
where
\[C_\delta:=1+\frac{\beta+\frac{9}{\delta}\tilde{C}_B}{2\alpha-\beta},\]
and $\eta$ is as in Proposition \ref{prop:inv}.
In particular, $u\in L^2(\Omega;L^\infty([0,T];H_1))$.
\end{corollary}
\begin{proof}
By Assumption~\ref{assu:inv}, $2\alpha-\beta>0$. Since $1+\frac{9}{\delta}\ge 1$, we have
$\tilde{\tilde{\eta}}\le\eta$, so the second estimate of Proposition~\ref{prop:inv} yields
\[
\mathbb{E}\int_0^t\|u(s)\|_1^2\,ds\le\frac{\|u_0\|_1^2+\eta t}{2\alpha-\beta}.
\]
Inserting this into the first estimate of Proposition~\ref{prop:inv} gives
\[
(1-\delta)\mathbb{E}\left(\sup_{0\le s\le t}\|u(s)\|_1^2\right)
\le\|u_0\|_1^2+\eta t+\frac{\beta+\frac{9}{\delta}\tilde{C}_B}{2\alpha-\beta}\left(\|u_0\|_1^2+\eta t\right)
=C_\delta\left(\|u_0\|_1^2+\eta t\right).
\]
The claim $u\in L^2(\Omega;L^\infty([0,T];H_1))$ follows.
\end{proof}

\subsection{Existence of invariant measures}\label{subsec:4.2}

We refer to \cite{DPZ:96} for details on Markovian semigroups for stochastic equations in Hilbert space and their invariant measures. Recall that the Markovian semigroup $\{P_t\}_{t\ge 0}$ associated to~\eqref{eq:gen_model} acts as follows
\begin{equation}\label{eq:Pt}
P_t \varphi(v):=\E[\varphi(u(t;v))]\quad 
\textrm{ for any }\quad \varphi\in \Bcal_b(H)\quad \textrm{ and }\quad v\in H.
\end{equation}
 For a semigroup $\{P_t\}_{t\ge 0}$, we define the dual semigroup $\{P_t^\ast\}_{t\ge 0}$ acting on $\Mcal_1(H)$ by
\begin{equation}\label{eq:Ptast}P_t^\ast \mu(A):=\int_H P_t \mathbbm{1}_A(x)\,\mu(\ud x)\quad \textrm{ for any }\quad A\in\Bcal(H),
\end{equation}
where $\mathbbm{1}_A$ denotes the indicator function of the set $A$.
For $v\in H$ and $A\in\Bcal(H)$, $t\ge 0$, set
\[P_t(v,A):=P_t^\ast \delta_v(A)=P_t\mathbbm{1}_A(v),\]
where $\delta_v$ denotes the Dirac delta measure at $v$.
Set also for $T\ge 0$, the time-average $\{Q_T\}_{T\ge 0}$ of $\{P_t\}_{t\ge 0}$,
\begin{equation}\label{eq:average}Q_T(v,A):=\frac{1}{T}\int_0^T P_t(v,A)\,dt.\end{equation}

\begin{definition}[Stochastically continuous Feller semigroup]
We say that $\{P_t\}_{t\ge 0}$ is a stochastically continuous Feller semigroup if for every $v\in H$ and every $r>0$ it follows that
\[\lim_{t\to 0} P_t(v,B(v,r))=1
\quad \textrm{ and }\quad
P_t(C_b(H))\subset C_b(H),
\]
where $B(v,r):=\{z\in H:\|z-v\|_H<r\}$.
\end{definition}

\begin{lemma}
Suppose that
Assumptions \ref{assu:space} and \ref{assu:noise} hold.
The transition semigroup $\{P_t\}_{t\geq 0}$ is a Markovian and stochastically continuous Feller semigroup on $H$.
\end{lemma}
\begin{proof}
The Markov property is proved in \cite[Proposition 4.3.5]{LR:15}.

Let us prove the stochastic continuity. Fix $v\in H$ and $r>0$. By
Theorem \ref{thm:existence_all_cases}, the solution $u(\cdot\,;v)$ has $\P$-a.s. continuous
paths with $u(0;v)=v$, so that $\|u(t;v)-v\|\to 0$ $\P$-a.s. as $t\to 0$. Since the
indicator functions $\mathbbm{1}_{\{\|u(t;v)-v\|<r\}}$ are bounded by $1$, dominated
convergence yields
\[P_t(v,B(v,r))=\P\big(\|u(t;v)-v\|<r\big)\longrightarrow 1\quad\text{as}\;\;t\to 0.\]

Let us prove the Feller property. Let $\varphi\in C_b(H)$ and $t\ge 0$. Then
$|P_t\varphi(v)|\le\E\big[|\varphi(u(t;v))|\big]\le\|\varphi\|_\infty$ for every $v\in H$, and
hence $\|P_t\varphi\|_\infty\le\|\varphi\|_\infty$. Let now $v_n,v\in H$, $n\in\N$, with
$\|v_n-v\|\to 0$ as $n\to\infty$. By \eqref{eq:Feller} below, which relies on Assumptions
\ref{assu:space} and \ref{assu:noise} only,
\[\E\|u(t;v_n)-u(t;v)\|^2\le\|v_n-v\|^2e^{Ct}\longrightarrow 0,\]
so that $u(t;v_n)\to u(t;v)$ in probability. Given any subsequence of $\{v_n\}$, we may thus
extract a further subsequence $\{v_{n_k}\}$ along which $u(t;v_{n_k})\to u(t;v)$ $\P$-a.s.
As $\varphi$ is continuous and bounded by $\|\varphi\|_\infty$, dominated convergence gives
$P_t\varphi(v_{n_k})\to P_t\varphi(v)$. Since every subsequence of $\{P_t\varphi(v_n)\}$
admits a further subsequence converging to $P_t\varphi(v)$, the whole sequence converges to
that limit, and therefore $P_t\varphi\in C_b(H)$.
\end{proof}

For the semigroup $\{P_t\}_{t\ge 0}$, defined in \eqref{eq:Pt}, and its dual semigroup $\{P^\ast_t\}_{t\ge 0}$, defined in \eqref{eq:Ptast}, we have the following notions.

\begin{definition}[Invariant measure]\label{def:inv}
A measure $\mu\in\Mcal_1(H)$ is said to be \emph{invariant} for the semigroup $\{P_t\}_{t\ge 0}$ if $P^\ast_t\mu=\mu$ for all $t\ge 0$.
\end{definition}

\begin{definition}[Ergodicity and mixing]\label{def:ergo}\phantom{ }
\begin{enumerate}
\item An invariant measure $\mu\in\Mcal_1(H)$ for the semigroup $\{P_t\}_{t\ge 0}$ is called \emph{ergodic} if
\[Q_T(v,\cdot) \rightharpoonup \mu\]
as $T\to\infty$ for any $v\in H$ in the sense of weak convergence of probability measures, where $\{Q_T\}_{T\ge 0}$ denotes the time-average of $\{P_t\}_{t\ge 0}$ defined in \eqref{eq:average}.
\item An invariant measure $\mu\in\Mcal_1(H)$ for the semigroup $\{P_t\}_{t\ge 0}$ is called \emph{strongly mixing} if
\[P_t(v,\cdot)\rightharpoonup \mu\]
as $t\to\infty$ for any $v\in H$ in the sense of weak convergence of probability measures.
\item An invariant measure $\mu\in\Mcal_1(H)$ for the semigroup $\{P_t\}_{t\ge 0}$ is called \emph{exponentially ergodic} if there exist $\lambda_0>0$ and a positive function $c:H\to\R$ such that for any $\varphi\in\Lip_b(H)$, all $t>0$, and all $v\in H$,
\[\left|P_t\varphi(v)-\int_H\varphi(z)\,\mu(dz)\right|\le c(v)e^{-\lambda_0 t}\Lip(\varphi).\]
\item An invariant measure $\mu\in\Mcal_1(H)$ for the semigroup $\{P_t\}_{t\ge 0}$ is called \emph{exponentially mixing} if there exist $\lambda_1>0$ and a positive function $c:H\to\R$ such that for all $t>0$, and all $v\in H$,
\[\|P_t(v,\cdot)-\mu\|_{\mathrm{FM}}\le c(v)e^{-\lambda_1 t},\]
where
\[\|\nu\|_{\mathrm{FM}}:=\sup\left\{\left|\int_H\varphi(z)\,\nu(dz)\right|\;\colon\;\varphi\in \Lip_b(H),\;\|\varphi\|_{\Lip_b}\le 1\right\}\]
is the \emph{Fortet--Mourier norm} (also called \emph{bounded Lipschitz norm}) of a signed measure $\nu$ on $\Bcal(H)$.
\end{enumerate}
\end{definition}

\begin{remark}
By the results of \cite[Section 3.5, pp. 39--40]{DPZ:96}, (i) implies the existence of a unique invariant measure; the converse requires an additional tightness condition. Moreover, (ii) implies (i), (iii) implies (iv), and both (iii) and (iv) imply (ii).
\end{remark}

\begin{proposition}[Tightness of measures]\label{prop:tight}
Suppose that
Assumptions \ref{assu:space}, \ref{assu:noise}, \ref{assu:hat}, and \ref{assu:inv} hold.
Let $v\in H_1$ and denote by $u(t;v)$ the unique solution to \eqref{eq:gen_model} with $u(0;v)=v$. Then, the family of probability measures $\{Q_t(v,\cdot)\}_{t\ge 1}$ defined in \eqref{eq:average}
is tight.
\end{proposition}
\begin{proof}
Set $\gamma:=2\alpha-\beta>0$, where $\beta$ is as in Assumption \ref{assu:inv}. Note that by Lemma \ref{lem:compact} the set
\begin{align*}
K_{R}:=\{\gamma\|\cdot\|^2_1\leq R\}\subset H
\end{align*}
is compact in $H$, for any radius $R>0$. By the Markov inequality,
\begin{align*}
Q_t(v,K_R) &= \frac{1}{t}\int_0^t\mathbb{P}(u(s;v)\in K_R)\, ds\\
&=\frac{1}{t}\int_0^t\mathbb{P}(\{\gamma\|u(s;v)\|_1^2\leq R\})\, ds\\
&=\frac{1}{t}\int_0^t\left( 1 - \mathbb{P}(\{\gamma\|u(s;v)\|_1^2> R\})\right)\, ds\\
&\geq \frac{1}{t}\int_0^t\left(1 - \frac{1}{R}\mathbb{E}\left(\gamma\|u(s;v)\|_1^2\right)\right)\, ds\\
&= 1 - \frac{1}{R}\left(\frac{1}{t}\mathbb{E}\left(\gamma\int_0^t\|u(s;v)\|_1^2\, ds\right)\right).
\end{align*}
From Proposition \ref{prop:inv}, we have that
\begin{align*}
\mathbb{E}\left(\frac{\gamma}{t}\int_0^t\|u(s;v)\|_1^2\, ds\right) \leq \frac{1}{t}\left( \|v\|_1^2 + \tilde{\tilde{\eta}} t\right)\leq \|v\|_1^2 + \tilde{\tilde{\eta}}. 
\end{align*}
Therefore, 
\begin{align*}
&1 - \frac{1}{R}\left(\frac{1}{t}\mathbb{E}\left(\gamma\int_0^t\|u(s;v)\|_1^2\, ds\right)\right)\\
&\geq 1 - \frac{1}{R}\left(\|v\|_1^2 + \tilde{\tilde{\eta}} \right)\\
& > 1-\varepsilon,
\end{align*}
if $R>\frac{1}{\varepsilon}(\|v\|_1^2 + \tilde{\tilde{\eta}})$. Thus, the family $\{Q_t\}_{t\geq 1}$ is tight. 
\end{proof}

\begin{theorem}[Existence of invariant measures]\label{thm:existence}
Suppose that
Assumptions \ref{assu:space}, \ref{assu:noise}, \ref{assu:hat}, and \ref{assu:inv} hold.
There exists at least one invariant probability measure $\mu$ for $\{P_t\}_{t\ge 0}$, that is,
\begin{align*}
\int_H (P_t\varphi )(v)\,\mu(dv) = \int_H\varphi(v)\,\mu(dv),\quad\text{for all}\;t\ge 0,
\end{align*}
for every $\varphi\in \mathcal{B}_b(H)$.

In particular, for every initial datum $v\in H_1$, there exists a sequence of real numbers $T_n\ge 1$, $n\in\N$ with $T_n\to\infty$ as $n\to\infty$, and a probability measure $\mu^v$ on the Borel sets of $L^2(U,\rho)$ which is invariant for the semigroup $\{P_t\}_{t\ge 0}$, and the following convergence holds
\[\frac{1}{T_n}\int_0^{T_n}\operatorname{Law}(u(t;v))\,dt\rightharpoonup \mu^v\quad\text{as}\;\;n\to\infty,\]
in the sense of weak convergence of probability measures, where $\operatorname{Law}(u(t;v))$ denotes the probability law of $u(t;v)$.

\end{theorem}
\begin{proof}
The tightness of the family of probability measures $\{Q_t(v,\cdot)\}_{t\ge 1}$ obtained in Proposition \ref{prop:tight} yields the existence of at least one invariant measure by the Krylov--Bogoliubov theorem \cite[Theorem 3.1.1]{DPZ:96}.
The second statement is the exact statement of the Krylov--Bogoliubov theorem.
\end{proof}

We thus have proved Theorem \ref{thm:main} (ii).

\subsection{Uniqueness of invariant measures}\label{subsec:4.3}

First, we prove the finiteness of second moments for any invariant measure under Assumption \ref{assu:ergo}.

\begin{lemma}\label{lem:2nd}
Suppose that
Assumptions \ref{assu:space}, \ref{assu:noise}, and \ref{assu:ergo} hold.
Then any invariant measure $\mu$ of $\{P_t\}_{t\ge 0}$ has second moments in the space $H$, in particular, there is some finite constant $\hat{C}\ge 0$ such that 
\begin{align*}
\int_H\|u\|^2\ \mu(du) \leq \hat{C}.
\end{align*}
\end{lemma}
\begin{proof}
Recall that by convention, if $f(0)=0$, we set $|f(0)|\|\rho\|^{1/2}_{L^1(U)}:=0$, even if $\rho\not\in L^1(U)$.
Fix $v\in H$. Applying It\^o's formula to the functional $u\mapsto\|u\|^2$, we obtain by Lemma \ref{lem:nemytskii_lipschitz},
\begin{align*}
&\|u(t;v)\|^2 + 2\alpha\int_0^t\|u(s;v)\|^2\, ds\\
\leq & \|v\|^2 + 2\int_0^t\|K\|_{L(H)}\left(\Lip(f)\|u(s;v)\|^2 + |f(0)|\|\rho\|_{L^1(U)}^{1/2}\|u(s;v)\|\right)\, ds\\
&+2\int_0^t\langle u(s;v),B(u(s;v))\, dW(s)\rangle +\int_0^t\|B(u(s;v))\|^2_{L_2(V,H)}\, ds.
\end{align*}
Define a sequence of stopping times
\[\lambda_N:=\inf\{s\in[0,T]:\|u(s;v)\|^2 > N \}\wedge T,\quad N\in\N,\]
and consider for $t\wedge\lambda_N$, noting that
\[t\mapsto\int_0^{t\wedge\lambda_N}\langle u(s;v),B(u(s;v))\, dW(s)\rangle\]
is a martingale with mean zero, for $\delta>0$ and some $C(\delta)>0$
\begin{align*}
&\E\|u(t\wedge\lambda_N;v)\|^2 + 2\alpha\E\int_0^{t\wedge\lambda_N}\|u(s;v)\|^2\, ds\\
\leq & \|v\|^2 + 2\E\int_0^{t\wedge\lambda_N}\|K\|_{L(H)}\left(\Lip(f)\|u(s;v)\|^2 + |f(0)|\|\rho\|_{L^1(U)}^{1/2}\|u(s;v)\|\right)\, ds\\
&+\E\int_0^{t\wedge\lambda_N}\|B(u(s;v))\|^2_{L_2(V,H)}\, ds\\
\leq & \|v\|^2 + \E\int_0^{t\wedge\lambda_N}(2\|K\|_{L(H)}\Lip(f)+\delta)\|u(s;v)\|^2 \,ds\\
&+2\|K\|_{L(H)}C(\delta)|f(0)|^2\|\rho\|_{L^1(U)}(t\wedge\lambda_N)\\
&+(C_B+\delta)\E\int_0^{t\wedge\lambda_N}\|u(s;v)\|^2\, ds+C(\delta)\|B(0)\|_{L_2(V,H)}^2(t\wedge\lambda_N).
\end{align*}
Recall that, by assumption,
$\lambda=2\|K\|_{L(H)}\Lip(f)+C_B<2\alpha$, and hence there exists $\delta>0$, such that $\lambda+2\delta<2\alpha$.
We obtain for some $C(\delta)>0$,
\begin{align*}
&\E\|u(t\wedge\lambda_N;v)\|^2 + (2\alpha-\lambda-2\delta)\E\int_0^{t\wedge\lambda_N}\|u(s;v)\|^2\, ds\\
\leq & \|v\|^2 + \left(2\|K\|_{L(H)}C(\delta)|f(0)|^2\|\rho\|_{L^1(U)}+C(\delta)\|B(0)\|_{L_2(V,H)}^2\right)(t\wedge\lambda_N).
\end{align*}
Set
\[\tilde{C}:=\left(2\|K\|_{L(H)}C(\delta)|f(0)|^2\|\rho\|_{L^1(U)}+C(\delta)\|B(0)\|_{L_2(V,H)}^2\right),\]
and
\[\tilde{\gamma}:=(2\alpha-\lambda-2\delta)>0.\]
Let $N\to\infty$, so that for any $t\in [0,T]$,
\begin{align*}
&\E\|u(t;v)\|^2 + \tilde{\gamma}\E\int_0^{t}\|u(s;v)\|^2\, ds
\leq  \|v\|^2 + \tilde{C}t.
\end{align*}
Let $L>0$ and set $\varphi_L:=\|\cdot\|^2\wedge L$, so that
$\varphi_L\in\Bcal_b(H)$ with $0\le\varphi_L\le L$. Discarding the first term on the
left-hand side and using $\varphi_L\le\|\cdot\|^2$, Tonelli's theorem and \eqref{eq:Pt}
give
\begin{equation}\label{eq:2nd-ptwise}
\tilde{\gamma}\int_0^t P_s\varphi_L(z)\, ds
=\tilde{\gamma}\,\E\int_0^t\varphi_L(u(s;z))\, ds\le \|z\|^2+\tilde{C}t,
\quad z\in H,\;t\ge 0.
\end{equation}
Let $R>0$ and denote $B_R:=\{z\in H\;\colon\;\|z\|\le R\}$. Integrating
\eqref{eq:2nd-ptwise} over $B_R$ with respect to $\mu$ yields
\begin{align*}
\tilde{\gamma}\int_0^t\int_{B_R}P_s\varphi_L(z)\,\mu(dz)\, ds
\le\int_{B_R}\|z\|^2\,\mu(dz)+\tilde{C}t\le R^2+\tilde{C}t,
\end{align*}
the right-hand side being finite for every fixed $R$. By the invariance of $\mu$ we have
$\int_H P_s\varphi_L\,d\mu=\int_H\varphi_L\,d\mu$ for every $s\ge 0$, and since
$0\le P_s\varphi_L\le L$,
\begin{align*}
\int_{B_R}P_s\varphi_L\,d\mu
=\int_H P_s\varphi_L\,d\mu-\int_{H\setminus B_R}P_s\varphi_L\,d\mu
\ge\int_H\varphi_L\,d\mu-L\,\mu(H\setminus B_R).
\end{align*}
Combining the last two inequalities, dividing by $t\ge 1$ and letting $t\to\infty$, we obtain
\begin{align*}
\tilde{\gamma}\left[\int_H\varphi_L\,d\mu-L\,\mu(H\setminus B_R)\right]\le\tilde{C}.
\end{align*}
As $\mu$ is a probability measure, $\mu(H\setminus B_R)\to 0$ as $R\to\infty$, whence
$\tilde{\gamma}\int_H\varphi_L\,d\mu\le\tilde{C}$. Finally, $\varphi_L\uparrow\|\cdot\|^2$
pointwise as $L\to\infty$, so that by the monotone convergence theorem
\begin{align*}
\int_H\|z\|^2\,\mu(dz)\le\frac{\tilde{C}}{\tilde{\gamma}}=:\hat{C}.
\end{align*}
Hence, $\mu\in\mathcal{M}_1^2(H)$.
\end{proof}

Let us prove the uniqueness of invariant measures, and the exponential ergodicity.

\begin{theorem}[Uniqueness of the invariant measure]\label{thm:uniqueness}
Suppose that
Assumptions \ref{assu:space}, \ref{assu:noise}, \ref{assu:hat}, \ref{assu:inv}, and \ref{assu:ergo} hold.
There exists a unique exponentially ergodic and exponentially mixing invariant probability measure $\mu$ for $\{P_t\}_{t\ge 0}$, and, in particular, for every $v\in H$,
\[\operatorname{Law}(u(t;v))\rightharpoonup \mu\quad\text{as}\;\;t\to\infty,\]
in the sense of weak convergence of probability measures, where $\operatorname{Law}(u(t;v))$ denotes the probability law of $u(t;v)$.
Moreover, for any $v\in H$, and any $t\ge 0$,
\[\left|P_t\varphi(v) - \int_H\varphi(z)\, \mu(dz) \right|\le \sqrt{2}\Lip(\varphi)\left(\|v\|^2+\hat{C}\right)^{1/2}e^{-\frac{2\alpha-\lambda}{2}t},\]
for any $\varphi\in\Lip_b(H)$, where
$\lambda=2\|K\|_{L(H)}\Lip(f)+C_B$ is as in Assumption \ref{assu:ergo}, and $\hat{C}\ge 0$ is as in Lemma \ref{lem:2nd}.
\end{theorem}
\begin{proof}
The existence of at least one invariant measure follows from Theorem \ref{thm:existence}. Let $v,z\in H$ and for any $\varphi\in\Lip_b(H)$ consider the difference
\begin{align*}
&\left|P_t\varphi(v)-P_t\varphi(z)\right|\\
\leq& \mathbb{E}\left|\varphi(u(t;v))-\varphi(u(t;z))\right|\\
\leq&  \Lip(\varphi)\mathbb{E}\|u(t;v)-u(t;z)\|\\
\leq& \Lip(\varphi)\left(\mathbb{E}\left(\|u(t;v)-u(t;z)\|^2\right)\right)^{1/2}.
\end{align*}
Next, we need to bound the term $\mathbb{E}\|u(t;v)-u(t;z)\|^2$, so let $\{\kappa_N\}$ be a sequence of stopping times given by
\begin{align*}
\kappa_N:=\inf\{t\in[0,T]: \|u(t;v)-u(t;z)\|^2>N\}\wedge T,\quad N\in\N,
\end{align*}
and apply It\^o's formula to obtain, for $t\wedge\kappa_N$,
\begin{align*}
&\mathbb{E}\|u(t\wedge\kappa_N;v)-u(t\wedge\kappa_N;z)\|^2 + 2\alpha\mathbb{E}\int_0^{t\wedge\kappa_N}\|u(s;v)-u(s;z)\|^2\, ds \\
=& \|v-z\|^2
 + 2\mathbb{E}\int_0^{t\wedge\kappa_N}\langle KF(u(s;v))-KF(u(s;z)),u(s;v)-u(s;z) \rangle\, ds\\
& + 2\mathbb{E}\int_0^{t\wedge\kappa_N}\langle u(s;v)-u(s;z),[B(u(s;v))-B(u(s;z))]\, dW(s) \rangle\\
& + \mathbb{E}\int_0^{t\wedge\kappa_N}\operatorname{Tr}_H [(B(u(s;v))-B(u(s;z)))((B(u(s;v)))^*-(B(u(s;z)))^*)]\, ds.
\end{align*}
Since 
\begin{align*}
M:t\mapsto \int_0^t\langle u(s;v)-u(s;z),[B(u(s;v))-B(u(s;z))]\, dW(s) \rangle,
\end{align*}
is a local martingale, we have that
\begin{align*}
\mathbb{E}\left[M_{t\wedge\kappa_N}\right] = \mathbb{E}\left[M_0\right] = 0. 
\end{align*}
Furthermore, 
\begin{align*}
&2\mathbb{E}\int_0^{t\wedge\kappa_N}\langle KF(u(s;v))-KF(u(s;z)),u(s;v)-u(s;z) \rangle\, ds\\
\leq& 2\mathbb{E}\int_0^{t\wedge\kappa_N}\|KF(u(s;v))-KF(u(s;z))\|  \|u(s;v)-u(s;z)\|\, ds\\
\leq& 2\mathbb{E}\int_0^{t\wedge\kappa_N}\|K\|_{L(H)}\Lip(f)\|u(s;v)-u(s;z)\| \|u(s;v)-u(s;z)\|\, ds\\
=& 2\|K\|_{L(H)}\Lip(f)\mathbb{E}\int_0^{t\wedge\kappa_N}\|u(s;v)-u(s;z)\|^2\, ds.
\end{align*}
Furthermore,
\begin{align*}
&\mathbb{E}\int_0^{t\wedge\kappa_N}\operatorname{Tr}_H[(B(u(s;v))-B(u(s;z)))((B(u(s;v)))^*-(B(u(s;z)))^*)]\, ds\\
\leq& \mathbb{E}\int_0^{t\wedge\kappa_N}\|B(u(s;v))-B(u(s;z))\|^2_{L_2(V,H)}\, ds\\
\leq& C_B\mathbb{E}\int_0^{t\wedge\kappa_N}\|u(s;v)-u(s;z)\|^2\, ds
\end{align*}
Collecting all the terms gives
\begin{align*}
&\mathbb{E}\|u(t\wedge\kappa_N;v)-u(t\wedge\kappa_N;z)\|^2 + 2\alpha\mathbb{E}\int_0^{t\wedge\kappa_N}\|u(s;v)-u(s;z)\|^2\, ds\\
\leq& \|v-z\|^2
 + \left(2\|K\|_{L(H)}\Lip(f)+C_B\right)\mathbb{E}\int_0^{t\wedge\kappa_N}\|u(s;v)-u(s;z)\|^2\, ds ,
\end{align*}
and thus
\begin{align*}
&\mathbb{E}\|u(t\wedge\kappa_N;v)-u(t\wedge\kappa_N;z)\|^2\\ 
\leq&  \|v-z\|^2
+ \left(\lambda - 2\alpha\right)\mathbb{E}\int_0^{t\wedge\kappa_N}\|u(s;v)-u(s;z)\|^2\, ds,
\end{align*}
where $\lambda=2\|K\|_{L(H)}\Lip(f)+C_B$ is as in Assumption \ref{assu:ergo}. Next, let $N\rightarrow\infty$, and note that by Gr\"onwall's lemma
\begin{equation}\label{eq:Feller}
\mathbb{E}\|u(t;v)-u(t;z)\|^2
\leq  \|v-z\|^2 e^{((\lambda - 2\alpha)\vee 0)t},\quad t\in [0,T].
\end{equation}
Now, to improve this, we apply It\^o's formula to the functional $(s,u)\mapsto e^{(2\alpha-\lambda)s}\|u\|^2$ instead, evaluated along the difference $u(s;v)-u(s;z)$. Bounding the drift exactly as above, its contribution is bounded above by
\begin{align*}
e^{(2\alpha-\lambda)s}\left[(2\alpha-\lambda)-2\alpha+2\|K\|_{L(H)}\Lip(f)+C_B\right]\|u(s;v)-u(s;z)\|^2 = 0,
\end{align*}
by the definition of $\lambda$. Hence, localising with $\kappa_N$, taking expectations, and letting $N\to\infty$ by Fatou's lemma,
\begin{align*}
\mathbb{E}\|u(t;v)-u(t;z)\|^2 \leq \|v-z\|^2 e^{(\lambda-2\alpha)t},\quad t\ge 0,
\end{align*}
and the estimate is independent of $T$.
Hence, under the assumption that $\lambda<2\alpha$,
\begin{align*}
&\mathbb{E}\|u(t;v)-u(t;z)\|^2 \leq \|v-z\|^2 e^{(\lambda-2\alpha)t}\longrightarrow 0\quad\text{as}\quad t\rightarrow\infty.
\end{align*}
Let $\mu$ be any invariant probability measure for $\{P_t\}_{t\ge 0}$. Then, for any $v\in H$, if $\mu$ has second moments, which is proved in Lemma \ref{lem:2nd} above, there exists $\hat{C}>0$ such that
\begin{equation}
\begin{split}
&\left|P_t\varphi(v) - \int_H\varphi(z)\, \mu(dz) \right|\\
\leq& \int_H|P_t\varphi(v)-P_t\varphi(z)|\, \mu(dz)\\
\leq& \Lip(\varphi)\int_H\left(\mathbb{E}\left[\|u(t;v)-u(t;z)\|^2\right]\right)^{1/2}\, \mu(dz)\\
\leq&\Lip(\varphi)\left(\int_H\|v-z\|^2\, \mu(dz)\right)^{1/2} e^{-\frac{2\alpha-\lambda}{2}t}\\
\leq&\sqrt{2}\Lip(\varphi)\left(\|v\|^2+\int_H\|z\|^2\, \mu(dz)\right)^{1/2} e^{-\frac{2\alpha-\lambda}{2}t}\\
\leq&\sqrt{2}\Lip(\varphi)\left(\|v\|^2+\hat{C}\right)^{1/2} e^{-\frac{2\alpha-\lambda}{2}t}\longrightarrow 0\quad\text{as}\;\; t\rightarrow\infty.
\end{split}
\end{equation}
Thus
\begin{align*}
\lim_{t\rightarrow\infty}P_t\varphi(x) = \int_H\varphi(z)\, \mu(dz),\quad x\in H,\quad \varphi\in \Lip_b(H),
\end{align*}
and thus $\mu$ is unique, exponentially ergodic and exponentially mixing, see Definition \ref{def:ergo} and \cite{DPZ:96}.
\end{proof}

We have proved Theorem \ref{thm:main} (iii).

\subsection{The monotone case}\label{subsec:4.4}

In this subsection, we assume that $K$ is self-adjoint, that is,
$\rho(x)w(x,y)=\rho(y)w(y,x)$ for a.e. $x,y\in U$, so that $K=\hat{K}$ and $\check{K}=0$;
for $\rho\equiv 1$ this means that $w$ is symmetric. We also assume that $K$ is
non-positive semidefinite, that is,
\begin{equation}\label{eq:negdef}\langle \hat{K}u,u\rangle\le 0\quad\text{for every}\;\;u\in H.\end{equation}
This kind of condition is related to inhibition effects in the neural field, see \cite{Bressloff,neuralfields2014ch9}.
Assume also that $f:\R\to\R$ is non-decreasing, which is equivalent to
\begin{equation}\label{eq:mono}(f(t)-f(s))(t-s)\ge 0\quad\text{for every}\;\;t,s\in\R.\end{equation}
All the activation functions in Example \ref{ex:activation} satisfy this condition. It follows easily that then for any $u,v\in H$,
\begin{equation}\label{eq:mono2}\langle F(u)-F(v),u-v\rangle\ge 0.\end{equation}

In this case, we obtain the following complementary result, in which the rate is measured in the $H_1$-norm and the initial datum is taken in $H_1$.  
\begin{theorem}\label{thm:monotone}
Suppose that
Assumptions \ref{assu:space}, \ref{assu:noise}, \ref{assu:hat}, and \ref{assu:inv} hold.
Under the additional assumptions \eqref{eq:negdef} and \eqref{eq:mono}, and under the assumption that
\[\bar{C}_B:=\hat{C}_B^2<2\alpha,\]
where $\hat{C}_B\ge 0$ is as in Assumption \ref{assu:inv},
there exists an exponentially ergodic invariant measure $\mu$ with second moments on $H_1$ which is unique in the class $\{\nu\in\mathcal{M}_1(H)\;\colon\;\nu(H_1)=1\}$, and there exists a constant $C>0$, such that for any $v\in H_1$, and any $t\ge 0$,
\[\left|P_t\varphi(v) - \int_{H_1}\varphi(z)\, \mu(dz) \right|\le \sqrt{2}\Lip(\varphi)\left(\|v\|^2_1+C\right)^{1/2}e^{-\frac{2\alpha-\bar{C}_B}{2}t},\]
for any $\varphi\in\Lip_b(H_1)$.
\end{theorem}
\begin{proof}
Extend $\|\cdot\|_1$ to all of $H$ by setting $\|z\|_1:=+\infty$ for $z\in H\setminus H_1$.
By Lemma \ref{lem:compact}, the sublevel sets $\{z\in H\;\colon\;\|z\|_1\le R\}$, $R>0$, are
compact in $H$, and in particular closed; consequently $z\mapsto\|z\|_1^2$ is a
$[0,\infty]$-valued lower semicontinuous function on $H$, and
$H_1=\bigcup_{N\in\N}\{z\in H\;\colon\;\|z\|_1\le N\}$ is $\sigma$-compact and hence a Borel
subset of $H$.

Fix $v_0\in H_1$. By Theorem \ref{thm:existence}, there exist $T_n\to\infty$ and an invariant
probability measure $\mu$ for $\{P_t\}_{t\ge 0}$ such that
$Q_{T_n}(v_0,\cdot)\rightharpoonup\mu$ as $n\to\infty$. Setting $\gamma:=2\alpha-\beta>0$,
Proposition \ref{prop:inv} and Tonelli's theorem give, for every $t\ge 1$,
\[\int_H\|z\|_1^2\,Q_t(v_0,dz)=\frac{1}{t}\,\E\int_0^t\|u(s;v_0)\|_1^2\, ds
\le\frac{\|v_0\|_1^2+\tilde{\tilde{\eta}}t}{\gamma t}
\le\frac{\|v_0\|_1^2+\tilde{\tilde{\eta}}}{\gamma}=:C.\]
Since $z\mapsto\|z\|_1^2$ is non-negative and lower semicontinuous, the portmanteau theorem
yields
\[\int_H\|z\|_1^2\,\mu(dz)\le\liminf_{n\to\infty}\int_H\|z\|_1^2\,Q_{T_n}(v_0,dz)\le C<\infty .\]
In particular $\|z\|_1<\infty$ for $\mu$-a.e. $z\in H$, that is $\mu(H_1)=1$, and $\mu$ has
second moments in $H_1$ with $\int_{H_1}\|z\|_1^2\,\mu(dz)\le C$.

The proof of Theorem \ref{thm:uniqueness} can be repeated verbatim, noting that the non-positive semidefiniteness and the monotonicity can be employed. Let $v,z\in H_1$.
We need to bound the term $\mathbb{E}\|u(t;v)-u(t;z)\|^2_1$, so let $\{\kappa_N\}$ be a sequence of stopping times given by
\begin{align*}
\kappa_N:=\inf\{t\in[0,T]: \|u(t;v)-u(t;z)\|^2_1>N\}\wedge T,\quad N\in\N,
\end{align*}
and apply It\^o's formula to obtain, for $t\wedge\kappa_N$,
\begin{align*}
&\mathbb{E}\|u(t\wedge\kappa_N;v)-u(t\wedge\kappa_N;z)\|^2_1 + 2\alpha\mathbb{E}\int_0^{t\wedge\kappa_N}\|u(s;v)-u(s;z)\|^2_1\, ds \\
=& \|v-z\|^2_1
 + 2\mathbb{E}\int_0^{t\wedge\kappa_N}\langle KF(u(s;v))-KF(u(s;z)),u(s;v)-u(s;z) \rangle_1\, ds\\
& + 2\mathbb{E}\int_0^{t\wedge\kappa_N}\langle u(s;v)-u(s;z),[B(u(s;v))-B(u(s;z))]\, dW(s) \rangle_1\\
& + \mathbb{E}\int_0^{t\wedge\kappa_N}\operatorname{Tr}_{H_1} [(B(u(s;v))-B(u(s;z)))((B(u(s;v)))^*-(B(u(s;z)))^*)]\, ds.
\end{align*}
Since 
\begin{align*}
M:t\mapsto \int_0^t\langle u(s;v)-u(s;z),[B(u(s;v))-B(u(s;z))]\, dW(s) \rangle_1,
\end{align*}
is a local martingale, we have that
\begin{align*}
\mathbb{E}\left[M_{t\wedge\kappa_N}\right] = \mathbb{E}\left[M_0\right] = 0. 
\end{align*}
Furthermore, since $K=\hat{K}$, by non-positive semidefiniteness \eqref{eq:negdef} and monotonicity \eqref{eq:mono2},
\begin{align*}
&2\mathbb{E}\int_0^{t\wedge\kappa_N}\langle KF(u(s;v))-KF(u(s;z)),u(s;v)-u(s;z) \rangle_1\, ds\\
=& 2\mathbb{E}\int_0^{t\wedge\kappa_N}\langle (-K)^{-1} (KF(u(s;v))-KF(u(s;z))),u(s;v)-u(s;z) \rangle\, ds\\
=& -2\mathbb{E}\int_0^{t\wedge\kappa_N}\langle F(u(s;v))-F(u(s;z)),u(s;v)-u(s;z) \rangle\, ds\\
\le& 0.
\end{align*}
Furthermore,
\begin{align*}
&\mathbb{E}\int_0^{t\wedge\kappa_N}\operatorname{Tr}_{H_1}[(B(u(s;v))-B(u(s;z)))((B(u(s;v)))^*-(B(u(s;z)))^*)]\, ds\\
\leq& \mathbb{E}\int_0^{t\wedge\kappa_N}\|B(u(s;v))-B(u(s;z))\|^2_{L_2(V,H_1)}\, ds\\
\leq& \bar{C}_B\mathbb{E}\int_0^{t\wedge\kappa_N}\|u(s;v)-u(s;z)\|^2_1\, ds.
\end{align*}
Collecting all the terms and letting $N\to\infty$, gives by Fatou's lemma,
\begin{align*}
&\mathbb{E}\|u(t;v)-u(t;z)\|^2_1 + (2\alpha-\bar{C}_B)\mathbb{E}\int_0^{t}\|u(s;v)-u(s;z)\|^2_1\, ds\\
\leq& \|v-z\|^2_1,
\end{align*}
Applying It\^o's formula instead to $(s,u)\mapsto e^{(2\alpha-\bar{C}_B)s}\|u\|_1^2$, evaluated along the difference $u(s;v)-u(s;z)$, the drift contribution is bounded above by
\begin{align*}
e^{(2\alpha-\bar{C}_B)s}\left[(2\alpha-\bar{C}_B)-2\alpha+\bar{C}_B\right]\|u(s;v)-u(s;z)\|^2_1 = 0,
\end{align*}
so that, localising with $\kappa_N$, taking expectations, and letting $N\to\infty$ by Fatou's lemma,
\begin{align*}
\mathbb{E}\|u(t;v)-u(t;z)\|^2_1 \leq \|v-z\|^2_1 e^{(\bar{C}_B-2\alpha)t},\quad t\ge 0.
\end{align*}
The proof is thus completed with the same arguments as in the proof of Theorem \ref{thm:uniqueness},
noting that one can infer the existence of second moments in $H_1$ with the help of Proposition \ref{prop:inv} exactly by the arguments of the proof of Lemma \ref{lem:2nd}.
\end{proof}

\section*{Declarations}
\noindent
\textbf{Acknowledgments.} 
JMT would like to thank Eva L\"ocherbach (\'Ecole Polytechnique) for fruitful discussions. Both authors would like to thank Christian Kuehn (Technische Universit\"at M\"unchen) for his interest and inspiring comments.

\noindent
\textbf{Funding.}
The research of AMO was supported by the Vilho, Yrj\"o and Kalle V\"ais\"al\"a Fund of the Finnish
Academy of Science and Letters. The research of JMT was partially supported by a Una Europa Consortium Seed Funding Grant (UNA-Random) and by the European Union's Horizon Europe research and innovation programme under the Marie Sk\l{}odowska-Curie Actions Staff Exchanges (Grant agreement no.~101183168 -- LiBERA, Call: HORIZON-MSCA-2023-SE-01).

\noindent
\textbf{Disclaimer.}
Co-funded by the European Union. Views and opinions expressed are however those of the authors only and do not necessarily reflect those of
the European Union or the European Education and Culture Executive Agency (EACEA). Neither the European Union nor EACEA can be held responsible for them.

\noindent
\textbf{Competing interests.} The authors declare that they have no conflict of interest.

\noindent
\textbf{Availability of data and materials.} Data sharing not applicable to this article as no datasets were generated or analyzed during the current study.

\end{document}